\documentclass[10pt]{amsart}
\usepackage{amsmath}
\usepackage[english,  activeacute]{babel}
\usepackage[latin1]{inputenc}
\usepackage{amssymb}
\usepackage{amsthm}
\usepackage{graphics,graphicx}
\usepackage{array}
\usepackage{a4wide}
\usepackage{cite}
\allowdisplaybreaks
\usepackage{color, url}
\usepackage{float}

\theoremstyle{plain}
\newtheorem{theorem}{Theorem}

\newtheorem{lemma}[theorem]{Lemma}
\theoremstyle{definition}

\newtheorem{conjecture}[theorem]{Conjecture}

\date{\today}

\title[Equivalence of the descents statistic]{Equivalence of the descents statistic on some (4,4)-avoidance classes of permutations }

\begin{document}

\author[T. Mansour]{Toufik Mansour}
\address{Department of Mathematics, University of Haifa,
3498838 Haifa, Israel}
\email{tmansour@univ.haifa.ac.il}
\author[M. Shattuck]{Mark Shattuck}
\address{Department of Mathematics, University of Tennessee,
37996 Knoxville, TN}
\email{shattuck@math.utk.edu}

\begin{abstract} In this paper, we compute and demonstrate the equivalence of the joint distribution of the first letter and descent statistics on six avoidance classes of permutations corresponding to two patterns of length four.  This distribution is in turn shown to be equivalent to the distribution on a restricted class of inversion sequences for the statistics that record the last letter and number of distinct positive letters, affirming a recent conjecture of Lin and Kim.  Members of each avoidance class of permutations and also of the class of inversion sequences are enumerated by the $n$-th large Schr\"{o}der number and thus one obtains a new bivariate refinement of these numbers as a consequence.  We make use of auxiliary combinatorial statistics, special generating functions (specific to each class) and the kernel method to establish our results.  In some cases, we utilize the conjecture itself in a creative way to aid in solving the system of functional equations satisfied by the associated generating functions.
\end{abstract}

\maketitle

\section{Introduction}

A permutation $\pi=\pi_1\cdots \pi_n \in \mathcal{S}_n$ is said to \emph{contain} $\rho=\rho_1\cdots \rho_m \in \mathcal{S}_m$ where $m \leq n$ if there is some subsequence of $\pi$ that is order isomorphic to $\rho$. That is, there exist indices $i_1<i_2<\cdots<i_m$ such that $\pi_{i_a}>\pi_{i_b}$ if and only if $\rho_a>\rho_b$ for all $a$ and $b$.  Otherwise, $\pi$ is said to \emph{avoid} $\rho$.  In this context, $\rho$ is referred to as a \emph{pattern}.  We say that $\pi$ avoids the set $K$ of patterns if it avoids each pattern in $K$ and denote by $\mathcal{S}_n(K)$ the subset of $\mathcal{S}_n$ whose members avoid $K$.  The study of pattern avoidance in permutations has been as object of considerable attention in recent decades and the notion has been extended to several other finite discrete structures (see, e.g., \cite{Kit} and references contained therein).

An \emph{inversion}  within $\pi=\pi_1\cdots \pi_n \in \mathcal{S}_n$ is an ordered pair $(a,b)$ where $a,b \in [n]=\{1,2,\ldots,n\}$ with $a<b$ and $\pi_a>\pi_b$.  The \emph{inversion sequence} of $\pi$ is defined by $x=x_1\cdots x_n$, where $x_i$ for each $i \in [n]$ records the number of entries to the right of the letter $i$ and smaller than $i$.  For example, if $\pi=621543 \in \mathcal{S}_6$, then $x=010125$. Note that the characterizing property of such sequences $x$ is $0 \leq x_i \leq i-1$ for all $i$.  The systematic study of patterns in inversion sequences is a topic that has only recently been initiated in \cite{CMSW,MaSh}.

Let $I_n$ denote the set of all inversion sequences of length $n$ and $I_n(\geq,-,>)$ the subset of $I_n$ consisting of those $e=e_1\cdots e_n$ in which there exist no indices $i<j<k$ such that both $e_i \geq e_j$ and $e_i>e_k$ hold.  See \cite{LK} or \cite{MarSav} for an explanation of this notation.
By an \emph{ascent} (\emph{descent}) within a sequence $s=s_1s_2\cdots s_n$, we mean an index $i \in [n-1]$ such that $s_i<s_{i+1}$ ($s_i>s_{i+1}$, respectively).  Let $\text{asc}(s)$ ($\text{desc}(s)$) denote the number of ascents (descents) of $s$.  If $s$ is a sequence having non-negative integral entries, then let $\text{dist}(s)$ denote the number of distinct positive letters appearing in $s$.

Lin and Kim made the following conjecture which provides a connection between inversion sequences and pattern avoidance in permutations.
\begin{conjecture}(Lin and Kim \cite{LK}.)
Let $(\nu,\mu)$ be a pair of patterns of length four. Then
$$\sum_{\pi \in \mathcal{S}_n(\nu,\mu)}q^{\text{asc}(\pi)}v^{\text{last}(\pi)}=
\sum_{e \in I_n(\geq,-,>)}q^{\text{dist}(e)}v^{\text{last}(e)+1}$$
for the following six pairs $(\nu,\mu)$:
\begin{align*}
&(4231,3241),\, (4231,2431),\,(4231,3421),\, (2431,3241),\, (3421,2431),\, (3421,3241).
\end{align*}
\end{conjecture}
Here, we find it more convenient to deal with the reversals of the patterns in each pair and consider alternatively the joint distribution of the descents and first letter statistics on the left-hand side above.  Also, we will represent an inversion sequence $e=e_1\cdots e_n$ using positive instead of non-negative entries (which can be achieved by adding one to each entry) so that $1 \leq e_i \leq i$ for all $i$.  Then $\text{dist}(e)$ would need to be replaced by $\text{dist}(e)-1$ and $\text{last}(e)+1$ by $\text{last}(e)$ on the right side of Conjecture 1.

The main objective of this paper is then to prove the following result.

\begin{theorem}\label{introtheor}
We have
\begin{align}
\sum_{\pi \in \mathcal{S}_n(\sigma,\tau)}q^{\text{desc}(\pi)}v^{\text{first}(\pi)}=
\sum_{e \in I_n(\geq,-,>)}q^{\text{dist}(e)-1}v^{\text{last}(e)}, \qquad n \geq 1, \label{introtheore1}
\end{align}
if and only if $(\sigma,\tau)$ is one of the following pattern pairs:
\begin{align*}
&(1243,1324),\, (1243,1342),\,(1243,1423),\, (1324,1342),\, (1324,1423),\, (1342,1423),
\end{align*}
where the inversion sequence $e$ is expressed using positive integers.
\end{theorem}

Combining the results that are proven in the subsequent sections implies Theorem \ref{introtheor}; note that in several instances, the specific result in question follows from a more general one.  We remark that the ``only if'' direction in Theorem \ref{introtheor} follows  from numerical evidence (upon considering the case $n=6$), which eliminates all other possible pairs $(\sigma,\tau)$ of patterns of length four.

\begin{table}[htp]
\begin{tabular}{|l|l||l|l||l|l|} \hline
  Pattern Pair & Reference & Pattern Pair & Reference & Pattern Pair & Reference \\\hline\hline
  1243,1324  &  \text{Theorem}~\ref{Th1324x1243}&
  1243,1342  &  \text{Theorem}~\ref{Th1342x1243}&
  1243,1423   & \text{Theorem}~\ref{Th1423x1243}\\\hline
  1324,1342   &  \text{Theorem}~\ref{Th11324x1342}&
  1324,1423   &  \text{Theorem}~\ref{Rgenfun}&
  1342,1423   &   \text{Theorem}~\ref{th13421423a}\\\hline
  $I_n(\geq,-,>)$  &  \text{Theorem}~\ref{UTheor}&
    &  &
  &   \\\hline
 \end{tabular}
\caption{Places where specific cases of Conjecture 1 are proven.}\label{maintable}
\end{table}

Let $S_{n,k}$  be given by
$$S_{n,k}=S_{n,k-1}+2S_{n-1,k}-S_{n-1,k-1}, \qquad 1 \leq k \leq n-2,$$
with $S_{n,n}=S_{n,n-1}=S_{n,n-2}$ for $n\geq 3$ and $S_{1,1}=S_{2,1}=S_{2,2}=1$.
One can show that $\sum_{k=1}^nS_{n,k}=S_n$, the $n$-th large Schr\"{o}der number (see \cite[A006318]{Slo}).   In \cite{MaSh2}, it was proven that $|\mathcal{S}_{n,i}(\sigma,\tau)|=S_{n,i}$ holds for nine pairs of patterns of length four (and no others), including the six listed in Theorem \ref{introtheor}, where $\mathcal{S}_{n,i}(\sigma,\tau)$  denotes the subset of $\mathcal{S}_n(\sigma,\tau)$ whose members start with $i$.  Here, a more technical argument is required to deal with the joint distribution of the first letter and descents statistics on $\mathcal{S}_n(\sigma,\tau)$ for the various pairs $(\sigma,\tau)$ and does not reduce to prior arguments when $q=1$.  To establish Theorem \ref{introtheor}, we show in each case that both sides of  equality \eqref{introtheore1} have the same (ordinary) generating function by a computational approach.  To do so, we make use of the kernel method \cite{HouM}, along with some educated guessing (aided at times by the conjecture itself), to solve a system of functional equations satisfied by the associated generating functions in each case. In particular, it is shown for each $(\sigma,\tau)$ that the generating function over $n \geq 1$ of both sides of \eqref{introtheore1} is given by
\begin{align*}
&\frac{vx}{1-vqx}+\frac{xv(vqx-v-x)t(xv)}{2(vq^2x^2-vqx^2-qx+vx-v-x+1)(vqx-vx-1)}\\
&+\frac{(x+(qx^2+qx+3x^2-2x-1)v)vx}{2(vqx-vx-1)(vqx-1)(vq^2x^2-vqx^2-qx+vx-v-x+1)}\\
&-\frac{(2q^2x^2+3q^2x-qx^2-qx-3q+2x-1+(1-2q)(qx-1)vqx)v^3x^2}{2(vqx-vx-1)(vqx-1)(vq^2x^2-vqx^2-qx+vx-v-x+1)},
\end{align*}
where $t(x)=\sqrt{(1-2q)^2x^2-2x(1+2q)+1}$.
Note that the coefficient of $v^iq^j$ in the preceding expression yields a new refinement of $S_n$.  For other extensions of the Schr\"oder numbers, see, e.g., \cite{BSS,Kr,PSu,RSi,Rog}. For other recent results on various restricted classes of inversion sequences, see \cite{AC,Lin,LY,YL,Yan} and references contained therein.

The organization of this paper is as follows.  In the next section, we perform the enumeration on the right side of \eqref{introtheore1} using an additional parameter that we call the height of an inversion sequence.  Indeed, we are able to find the generating function of the joint distribution of the dist, last letter and height statistics on $I_n(\geq,-,>)$.  In the third and fourth sections, we treat the pattern pairs $(\sigma,\tau)=(1324,1423)$ and $(1342,1423)$, respectively, where we proceed by solving a more refined version of the problem in the first case and a generalization of the problem in the latter case.  In the final section, we treat the remaining four pattern pairs by considering in addition the second letter statistic, which is necessary for finding recurrences satisfied by the joint distribution.  These recurrences may then be rewritten in terms of generating functions, leading to a system of functional equations in each case.  We leave the question of finding bijective proofs of the statistical equivalences in Theorem \ref{introtheor} as an open problem.

\section{Statistics on a restricted inversion sequence}

In this section, we determine the joint distribution of the last letter and dist statistics on $\mathcal{U}_n=I_n(\geq,-,>)$.  By the \emph{height} of $e=e_1\cdots e_n \in I_n$, which will be denoted by $\text{hght}(e)$, we mean the greatest $x$ such that $x \geq y$, where $y$ is the successor of $x$.  That is, $\text{hght}(e)=\max\{e_i:e_i\geq e_{i+1} \text{ and } 1 \leq i \leq n-1\}$.  Note that one may restrict attention to $e_i$ corresponding to leftmost occurrences of letters in $e$ when determining $\text{hght}(e)$.

For example, if $n=12$ and $e=12\underline{33}45\underline{74}7789 \in I_{12}$, then $\text{hght}(e)=7$. Note that indeed $e \in \mathcal{U}_n$; to check this quickly, one need only verify that there are no letters in $[2]$ to the right of the second $3$ and no letters in $[6]$ to the right of the second $4$, where we have underlined the adjacencies in $e$ which determine the current height (for example, the height of the partial sequence $123345\in \mathcal{U}_6$ would be 3, which then becomes 7 after the addition of the next two letters 7,4).

Let $\mathcal{U}_{n}(i,j)$ denote the subset of $\mathcal{U}_n$ consisting of those $e$ such that $\text{last}(e)=i$ and $\text{hght}(e)=j$.  Note that $\mathcal{U}_n(i,j)$ can be nonempty only when $1 \leq i \leq n$ and $1 \leq j \leq n-1$.  Further, $e=12\cdots n$ belongs to no subset $\mathcal{U}_{n}(i,j)$ since there is no $x$ such that $x \geq y$ in the description above, and hence will be counted separately.  Observe further that, by the definitions, we have that $\mathcal{U}_n(n,n-1)$ is empty since the first $n-1$ would have to occur in the penultimate position, which implies the last letter must belong to $[n-1]$ in order for a height of $n-1$ to be achieved.

Given $n \geq 2$, $1 \leq i \leq n$ and $1 \leq j \leq n-1$, let $u_n(i,j)=u_n(i,j;q)$ be given by
$$u_n(i,j)=\sum_{e \in \mathcal{U}_n(i,j)}q^{\text{dist}(e)-1},$$
and put zero for $u_n(i,j)$ otherwise.  For example, we have
$$\mathcal{U}_5(2,4)=\{11142,11242,11342,12142,12242,12342\},$$
and thus $u_5(2,4)=4q^2+2q^3$.

The array $u_n(i,j)$ is determined recursively as follows.

\begin{lemma}\label{invseqLem}
If $n \geq 3$, then
\begin{equation}\label{uneqn1}
u_n(i,j)=\delta_{j,n-1}\cdot q^{n-2}+u_{n-1}(j,i)+q\sum_{k=1}^{i-1}(u_{n-2}(j-1,k)+u_{n-1}(j-1,k)), \quad 1\leq i<j\leq n-1,
\end{equation}
\begin{equation}\label{uneqn2}
u_n(i,j)=q\sum_{\ell=1}^{i-1}u_{n-1}(\ell,j), \quad 1 \leq j <i \leq n,
\end{equation}
and
\begin{equation}\label{uneqn3}
u_n(i,i)=\sum_{k=1}^{i-1}u_{n-1}(i,k)+\sum_{\ell=1}^{i}u_{n-1}(\ell,i), \quad 1 \leq i \leq n-2,
\end{equation}
with $u_3(2,2)=q$ and $u_n(n-1,n-1)=q^{n-2}+q\sum_{i=1}^{n-2}\sum_{j=1}^{n-3}u_{n-2}(i,j)$ for $n \geq 4$ and initial values $u_2(1,1)=1$, $u_2(2,1)=0$.
\end{lemma}
\begin{proof}
The initial values when $n=2$ follow from the definitions, so assume $n \geq 3$.  To show \eqref{uneqn1}, first note that for $e \in \mathcal{U}_{n}(i,j)$ where $i<j$, we claim that the final two letters of $e$ are $j,i$.  To realize this, note that $\text{hght}(w)=j$ for $w \in \mathcal{U}_n$ implies $w$ is expressible as $w=\alpha j\ell \beta$, where $1 \leq \ell \leq j$, $\alpha$ contains no letters greater than $j-1$ and $\beta$ contains no letters in $[j-1]$.  Then $e$ belonging to $\mathcal{U}_n(i,j)$ where $i<j$ implies that the section of $e$ corresponding to $\beta$ is empty, which yields the claim.  Let $e=e'i$, where we assume $e'\neq 12\cdots(n-1)$.   Note that $e'$ has all its letters in $[j-1]$ except for the terminal $j$, for otherwise, $xji$ would occur with $x\geq j>i$ and $j,i$ the final two letters of $e$, which is impossible.  Thus, $\text{hght}(e')=k$ for some $1 \leq k \leq j-1$.  The successor $t$ of the leftmost $k$ in $e'$ must then belong to $[k]$, for otherwise $\text{hght}(e')$ would be strictly greater than $k$.  Thus, we must have $k \leq i$, for otherwise $kti$ would form an occurrence of the forbidden pattern.

If $k=i$, then one may delete the terminal $i$ from $e$ resulting in a member of $\mathcal{U}_{n-1}(j,i)$ and hence there are $u_{n-1}(j,i)$ possibilities.  So assume $k<i$.  In this case, we have $e=\delta ktk^p\rho$, where $1 \leq t \leq k$, $p \geq 0$, $\delta$ is $(k-1)$-ary and $\rho$ contains no letters in $[k]$ and ends in $j,i$.  Further, we have $\rho=\rho^*ji$, where $\rho^*$ if nonempty contains letters in $[k+1,j-1]$ and is strictly increasing (for otherwise $\text{hght}(e')=k$ would be violated).  Then $i>k$ implies $i$ can only occur once or twice in $\rho$ (and hence in $e$).  If $i$ occurs once in $\rho$ (at the end), then deleting $i$ results in a member of $\mathcal{U}_{n-1}(j-1,k)$ after reduction of letters greater than $i$.  Considering all possible $k$ then gives $q\sum_{k=1}^{i-1}u_{n-1}(j-1,k)$, where the $q$ accounts for the additional distinct letter $i$.
If $i$ occurs twice in $\rho$ (once in $\rho^*$ and again at the end), then deleting both $i$ is seen to give
$q\sum_{k=1}^{i-1}u_{n-2}(j-1,k)$ possibilities.  Combining the previous cases on $k$ then implies \eqref{uneqn1} if $j<n-1$.  On the other hand, if $j=n-1$, then $e'=12\cdots(n-1)$ is also possible, and thus there is an additional member of $\mathcal{U}_{n}(i,n-1)$ of weight $q^{n-2}$, which completes the proof of \eqref{uneqn1} in all cases.

To show \eqref{uneqn2}, first note that for $w \in \mathcal{U}_n(i,j)$
where $j<i$, we must have $w=w'jtj^p\rho$, where $w'$ is $(j-1)$-ary, $1 \leq t \leq j$, $p \geq 0$ and $\rho$ is strictly increasing on $\{j+1,j+2,\ldots\}$.  Note that $i>j$ implies $\rho$ is nonempty. Then deleting the terminal $i$ of $w$ yields a member of $\mathcal{U}_{n-1}(\ell,j)$ for some $\ell$ since the leftmost $j$ (and its successor) are unaffected.  Note that $|\rho|>1$ implies $j < \ell \leq i-1$, whereas $|\rho|=1$ implies $1 \leq \ell \leq j$ (in the latter case, $\ell=j$ if and only if $p>0$ or $p=0$ with $t=j$).  Furthermore, $i>j$ implies the terminal $i$ is the only letter of its kind.  Thus considering all possible $\ell$ yields formula \eqref{uneqn2}.

To show \eqref{uneqn3}, let $w \in \mathcal{U}_n(i,i)$, where $1 \leq i \leq n-2$.  First suppose $w=w'ii$, where $w'$ is $(i-1)$-ary.  Then deleting the final $i$ from $w$ results in a member of $\mathcal{U}_{n-1}(i,k)$ for some $k<i$, which accounts for the first sum on the right side.  Note that $i \leq n-2$ implies $w'i \neq 12\cdots(n-1)$ and hence belongs to some $\mathcal{U}_{n-1}(i,k)$.  On the other hand, if the leftmost $i$ of $w$ occurs to the left of the penultimate position, then deleting the final $i$ results in a member of $\mathcal{U}_{n-1}(\ell,i)$ for some $\ell \in [i]$.  Note that $\ell<i$ in this last case if and only if the successor of the leftmost $i$ is less than $i$ and occurs in the penultimate position of $w$. Considering all possible $\ell$ yields the second sum and completes the proof of \eqref{uneqn3}.  If $i=n-1$, then $w \in \mathcal{U}_n(n-1,n-1)$ implies $w=w'(n-1)(n-1)$, where $w' \in \mathcal{U}_{n-2}$ with no restrictions.  Then there are $q^{n-2}+q\sum_{i=1}^{n-2}\sum_{j=1}^{n-3}u_{n-2}(i,j)$ possibilities in this case, with the extra factor of $q$ multiplying the sum accounting for the terminal letters $n-1$, which completes the proof.
\end{proof}

Using Lemma \ref{invseqLem}, one gets the following values of the array $u_n(i,j)$ when $n=3$ and $n=4$:
$$u_3(1,1)=1,\quad u_3(1,2)=u_3(2,1)=u_3(2,2)=u_3(3,1)=q,\quad u_3(3,2)=0$$
and
\begin{align*}
u_4(1,1)=1 && u_4(1,2)=q && u_4(1,3)=q+q^2\\
u_4(2,1)=q && u_4(2,2)=3q && u_4(2,3)=2q^2\\
u_4(3,1)=q+q^2 && u_4(3,2)=2q^2 && u_4(3,3)=q+q^2\\
u_4(4,1)=q+2q^2 && u_4(4,2)=2q^2 &&u_4(4,3)=0,
\end{align*}
which may be verified using the definitions.

Let $U_n(v,w)=U_n(v,w;q)$ be given by
$$U_n(v,w)=\sum_{i=1}^n\sum_{j=1}^{n-1}u_n(i,j)v^iw^j, \qquad n \geq 2,$$
with $U_1(v,w)=0$.  We wish to find $U_n(v,w)+v^{n}q^{n-1}$ for all $n \geq 1$ and, in particular, $U_n(v,1)+v^{n}q^{n-1}$, where the additional term accounts for $12\cdots n$.  Define $$U(x,v,w;q)=U(x,v,w)=\sum_{n\geq2}\sum_{i=1}^n\sum_{j=1}^{n-1}u_n(i,j)v^iw^jx^n,$$
where the $q$ argument is often suppressed.  Then $$U(x,v,1)+\sum_{n\geq 1}v^nq^{n-1}x^n=U(x,v,1)+\frac{vx}{1-vqx}$$
yields the generating function for the joint distribution on $\mathcal{U}_n$ for $n \geq 1$ of the last letter and dist statistics (marked by $v$ and $q$, respectively).

In order to find $U(x,v,w)$, we define the further generating functions $$U^+(x,v,w)=\sum_{n\geq3}\sum_{i=1}^{n-2}\sum_{j=i+1}^{n-1}u_n(i,j)v^iw^jx^n,$$ $$U^-(x,v,w)=\sum_{n\geq2}\sum_{i=2}^n\sum_{j=1}^{i-1}u_n(i,j)v^iw^jx^n$$ and $$U^0(x,v)=\sum_{n\geq2}\sum_{i=1}^{n-1}u_n(i,i)v^ix^n.$$
Clearly,
$$U(x,v,w)=U^+(x,v,w)+U^-(x,v,w)+U^0(x,vw),$$ by the definitions.
Rewriting \eqref{uneqn1}-\eqref{uneqn3} in terms of generating functions yields
\begin{align*}
U^+(x,v,w)&=\frac{vw^2qx^3}{(1-wqx)(1-vwqx)}+xU^-(x,w,v)+\frac{vwqx^2}{1-v}(U^-(x,w,v)-U^-(x,vw,1))\\
&+\frac{vwqx}{1-v}(U^-(x,w,v)-U^-(x,vw,1)-wqxU(wx,1,v)+vwqxU(vwx,1,1)),\\
U^-(x,v,w)&=\frac{vqx}{1-v}(U^-(x,v,w)-vU^-(vx,1,w)+U^0(x,vw)-vU^0(vx,w))\\
&+\frac{vqx}{1-v}(U^+(x,1,vw)-vU^+(vx,1,w)),\\
U^0(x,v)&=\frac{vx^2}{1-vqx}+xU^+(x,1,v)+xU^0(x,v)+xU^-(x,v,1).
\end{align*}

One may verify by programming the following solution of the foregoing system of functional equations, where $t(x)$ is as above.

\begin{theorem}\label{UTheor}
The generating function $U(x,v,w)$ is given by
$$U(x,v,w)=U^+(x,v,w)+U^-(x,v,w)+U^0(x,vw),$$
where
\begin{align*}
&U^0(x,v)=\frac{vx^2(2(v-1)vqx-(v-3)vx-3v+1)+v(v+1)x^2t(vx)}{2(2vqx-vx+v+x-1)(vqx-vx-1)},\\
&U^+(x,v,w)=-\frac{vw^2qx^2(vwqx-v+x)t(vwx)}{2(2vwqx-vwx+vw+x-1)(vw^2q^2x^2-vw^2qx^2-wqx+vwx-wx-v+1)}\\
&-\frac{vw^2qx^2((2vwq^2-vwq+1)vwx^2-(2v^2wq+vwq-v^2w+2vw+2v-1)x+v)}{2(2vwqx-vwx+vw+x-1)(vw^2q^2x^2-vw^2qx^2-wqx+vwx-wx-v+1)},\\
&U^-(x,v,w)=\frac{v^2wqx^2(vwqx-vwx+vw+w-1)t(vwx)}{2(2vwqx-vwx+vw+x-1)(v^2wq^2x^2-v^2wqx^2-vqx+vwx-vx-w+1)}\\
&+\frac{(-2v(q+1)x^2+(vq+2v+2)x-1)v^2wqx^2}{2(v^2wq^2x^2-v^2wqx^2-vqx+vwx-vx-w+1)(vqx-1)(2vwqx-vwx+vw+x-1)}\\
&+\frac{(2v^2q(q-1)x^3-(3vq^2-2(v+1))vx^2-(v^2q-2vq+2(v+1)^2)x+v+1)v^2w^2qx^2}{2(v^2wq^2x^2-v^2wqx^2-vqx+vwx-vx-w+1)(vqx-1)(2vwqx-vwx+vw+x-1)}\\
&+\frac{(2q-1)(vqx-1)(vqx-vx+v+1)v^3w^3qx^3}{2(v^2wq^2x^2-v^2wqx^2-vqx+vwx-vx-w+1)(vqx-1)(2vwqx-vwx+vw+x-1)}.
\end{align*}
\end{theorem}

As a corollary, we have
\begin{align*}
U(x,v,1)&=\frac{xv(vqx-v-x)t(xv)}{2(vq^2x^2-vqx^2-qx+vx-v-x+1)(vqx-vx-1)}\\
&+\frac{(x+(qx^2+qx+3x^2-2x-1)v)vx}{2(vqx-vx-1)(vqx-1)(vq^2x^2-vqx^2-qx+vx-v-x+1)}\\
&-\frac{(2q^2x^2+3q^2x-qx^2-qx-3q+2x-1+(1-2q)(qx-1)vqx)v^3x^2}{2(vqx-vx-1)(vqx-1)(vq^2x^2-vqx^2-qx+vx-v-x+1)},
\end{align*}
and, in particular,
$$U(x,1,1)=\frac{1-(1+3q)x+q(2q-1)x^2-(1-qx)\sqrt{(1-2q)^2x^2-2x(1+2q)+1}
}{2q(1+x-qx)(1-qx)}.$$

\section{The case $(1324,1423)$.}

To find a recursive structure leading to recurrences in this case, we make use of a generating tree approach (see, e.g., \cite{West,West2}).  Consider forming members of $\mathcal{R}_n=\mathcal{S}_n(1324,1423)$ from members of $\mathcal{R}_{n-1}$, expressed using the elements of $[2,n]$, by inserting $1$ appropriately.  By an \emph{active site} of $\pi \in \mathcal{R}_n$, written using letters in $[2,n+1]$, we mean a position in which one may insert $1$ without introducing $1324$ or $1423$.  One often refers to $\rho \in \mathcal{R}_{n-1}$ in which the element $1$ is inserted in forming $\pi \in \mathcal{R}_n$ as the \emph{precursor} of $\pi$.

Let $\text{act}(\pi)$ denote the number of (active) sites of $\pi$.  Then it is seen that $\text{act}(\pi)=n+1$ if and only if $\pi$ avoids both $213$ and $312$.  On the other hand, suppose that the $i$-th letter from the right within $\pi$, say $x$, starts a $213$ or $312$ and is the rightmost such letter of $\pi$ to do so (i.e., $i$ is minimal).  Then $\text{act}(\pi)=i$, where $3 \leq i \leq n$, with the sites of $\pi$ corresponding to the $i$ possible positions of $\pi$ to the right of $x$ in which to insert $1$.
Note that $\text{act}(\pi)=i$ where $3 \leq i \leq n$ implies $\pi$ may be decomposed as
$$\pi=\alpha x\beta\gamma,$$
where $\alpha$ is of length $n-i$, $\beta$ is increasing with $|\beta|\geq 2$ and contains $\max(\beta\cup \gamma)$, $x>\min(\beta)$ and $\gamma$ is decreasing and possibly empty.

By an \emph{active descent}, we will mean an active site that is itself a descent.  Note that in the decomposition of $\pi$ above, the active descents of $\pi$ would correspond to the positions directly preceding the letters of $\gamma$.  Let $\text{dact}(\pi)$ denote the number of active descents of $\pi$.  Define $\mathcal{R}_n(i,j)$ as the subset of $\mathcal{R}_n$ consisting of those $\pi$ for which $\text{act}(\pi)=i$ and $\text{dact}(\pi)=j$.  Note that $0 \leq j \leq i-2$, with $j=i-2$ occurring only when $\pi$ has no ascents.
Let $r_n(i,j)=r_n(i,j;p,q)$ be defined by
$$r_n(i,j)\sum_{\pi \in \mathcal{R}_n(i,j)}p^{\text{first}(\pi)}q^{\text{desc}(\pi)}.$$
For example,  we have
$$\mathcal{R}_5(4,0)=\{23145,32145,34125,35124,42135,43125,45123,52134,53124,54123\},$$
which implies $r_5(4,0)=p^2q(1+p)^2+p^3q^2(1+2p+3p^2)$.

It is seen that the $r_n(i,j)$ can assume nonzero values only for $3 \leq i \leq n+1$ and $0 \leq j \leq i-2$ if $n \geq 2$, with $r_1(2,0)=p$.  When $n=2$, we have $r_2(3,1)=p^2q$ and $r_2(3,0)=p$.  If $n=3$, then $r_3(4,0)=p$, $r_3(4,1)=pq(1+p)$, $r_3(4,2)=p^3q^2$ and $r_3(3,0)=p^2q(1+p)$.  Note that one may regard $p$ and $q$ as indeterminates or, alternatively, as fixed positive integers.  In the former case, one would regard members of the set $\mathcal{R}_n(i,j)$ as being weighted according to certain parameter values, whereas in the latter, one can enumerate ``colored'' members of $\mathcal{R}_n(i,j)$ wherein a first letter of size $\ell$ is assigned one of $p^\ell$ colors and each descent is assigned one of $q$ colors, independently of the others.  These interpretations of the polynomial $r_n(i,j)$ may be used interchangeably when the context is clear and analogous interpretations apply to other distribution polynomials encountered.

The $r_n(i,j)$ are given recursively as follows.

\begin{lemma}\label{1324x1423L1}
We have
\begin{equation}\label{1324x1423e1}
r_n(j+3,j)=pqr_{n-1}(j+2,j-1)+pq\sum_{i=j+3}^nr_{n-1}(i,j)+p\sum_{k=j+1}^{n-2}\sum_{i=k+2}^n r_{n-1}(i,k), \quad 0 \leq j \leq n-3,
\end{equation}
\begin{equation}\label{1324x1423e2}
r_{n}(i,j)=pqr_{n-1}(i-1,j-1)+pr_{n-1}(i-1,j)+pq\sum_{\ell=i}^nr_{n-1}(\ell,j), \quad 4 \leq j+4 \leq i \leq n,
\end{equation}
\begin{equation}\label{1324x1423e3}
r_n(j+2,j)=\delta_{j,n-1}\cdot p^nq^{n-1}, \quad 0 \leq j \leq n-1,
\end{equation}
and
\begin{equation}\label{1324x1423e4}
r_n(n+1,j)=\sum_{\ell=1}^{n-1}p^\ell q^j\binom{n-\ell-1}{n-j-2}, \quad 0 \leq j \leq n-2.
\end{equation}
\end{lemma}
\begin{proof}
Formula \eqref{1324x1423e3} follows from observing that $\mathcal{R}_n(j+2,j)$ is nonempty if and only if $j=n-1$, in which case it consists of the single element $n(n-1)\cdots 1$ of weight $p^nq^{n-1}$.  For \eqref{1324x1423e4},
observe first that $\mathcal{R}_{n}(n+1,j)$ consists of $\pi$ of the form $\pi=\alpha n \beta$, where $\alpha$ is increasing and $\beta$ is decreasing with $|\beta|=j$.  Then $0 \leq j \leq n-2$ implies $\alpha$ is nonempty and we enumerate members of $\mathcal{R}_n(n+1,j)$ in this case according to the first letter $\ell$ where $1 \leq \ell \leq n-1$.  Then $\alpha$ must contain $n-j-2$ members of $[\ell+1,n-1]$ in increasing order, with the remaining elements of $[n+1]-\{\ell\}$ comprising $\beta$.  Thus, there are $\binom{n-\ell-1}{n-j-2}$ members of $\mathcal{R}_n(n+1,j)$ with first letter $\ell$, each having weight $p^\ell q^j$.  Considering all possible $\ell$ then gives \eqref{1324x1423e4}.

To show \eqref{1324x1423e1}, suppose $\pi \in \mathcal{R}_n(j+3,j)$, where $0 \leq j \leq n-3$.  Then $\pi$ may be expressed as $\pi=\alpha xyz\beta$, where $\alpha$ or $\beta$ is possibly empty, $y<x,z$, and $\beta$ is decreasing of length $j$ with $\max(\beta)<z$ if $j>0$.  Note that then the $j+3$ sites of $\pi$ correspond to the positions to the right of the letter $x$ and that either $y=1$ or $1$ occurs at the end of $\pi$.  For if not, then $y<x,z$ and $\beta$ decreasing would imply $1$ lies in $\alpha$ ensuring an occurrence of $1324$ or $1423$, as witnessed by $1xyz$.  If $1$ occurs at the end of $\pi$, then there are $pqr_{n-1}(j+2,j-1)$ possibilities since the $1$ would increase all parameter values (including act and dact) by one.

So assume $1$ does nor occur at the end of $\pi$ and we consider cases based on the relative sizes of $1$'s neighbors as follows. Let us say that a pattern $\tau$ occurs as a \emph{subword} in $\pi$ if the letters corresponding to an occurrence of $\tau$ are consecutive entries of $\pi$.  We consider cases based on whether the element $1$ within $\pi \in \mathcal{R}_n(j+3,j)$ is the middle letter in an occurrence of a $213$ or $312$ subword.
In order for $1$ to be involved in a $213$, then $1$ must be inserted directly preceding the $(j+1)$-st letter from the right within a member of $\mathcal{R}_{n-1}(i,j)$ for some $j+3 \leq i \leq n$.  Note that this increases the first letter by one as well as the number of descents (as $1$ is placed within an ascent in this case).  Considering all possible $i$ then gives the first sum on the right side of \eqref{1324x1423e1}.  For $1$ to be involved in a $312$ instead, the precursor $\rho$ of $\pi$ must satisfy $\text{dact}(\rho)=k$ where $j+1 \leq k \leq n-2$, with $1$ being inserted into the $(j+1)$-st active descent of $\rho$ from the right.  Note that $\text{act}(\rho)=i$ for any $i \in [k+2,n]$ since there is no restriction as to the number of additional sites of $\rho$.  Then $\rho \in \mathcal{R}_{n-1}(i,k)$ where $i$ and $k$ are as specified, with the number of descents unchanged by the insertion of $1$. Considering all possible $i$ and $k$ then gives the second summation formula in and finishes the proof of \eqref{1324x1423e1}.

A similar argument may be given for \eqref{1324x1423e2}.  First note that there are $pqr_{n-1}(i-1,j-1)$ possible $\pi$ obtained by inserting $1$ at the end of the precursor $\rho$ and $pr_{n-1}(i-1,j)$ possibilities if $1$ is inserted in the leftmost site of $\rho$ (in which case $1$ would be part of a $312$ subword).  On the other hand, if $1$ is to be part of a $213$, then it must be inserted into $\rho \in \mathcal{R}_{n-1}(\ell,j)$ for some $i \leq \ell \leq n$ directly following the $(i-1)$-st letter from the right.  Note that $j+3<i \leq \ell$ implies that there is an ascent between the $(i-1)$-st and $(i-2)$-nd rightmost letters of $\rho$.  Hence, insertion of $1$ into this site introduces an additional descent, implying that there are $pq\sum_{\ell=i}^nr_{n-1}(\ell,j)$ possible $\pi$ in which $1$ forms a $312$.  Combining this case with the previous implies \eqref{1324x1423e2} and completes the proof.
\end{proof}

By Lemma \ref{1324x1423L1}, one obtains the following nonzero values for $r_n(i,j)$ when $n=4$:
\begin{align*}
r_4(3,0)=(2p^2+p^3)q+(p^3+2p^4)q^2,&& r_4(4,0)=(p^2+p^3+p^4)q, && r_4(5,0)=p,\\
r_4(4,1)=(p^2+2p^3+2p^4)q^2, && r_4(5,1)=(2p+p^2)q, && r_4(5,2)=(p+p^2+p^3)q^2,\\
r_4(5,3)=p^4q^3, &&~ && ~
\end{align*}
which may be verified directly using the definitions.

Define
$$R(x,v,w;p,q)=R(x,v,w)=\sum_{n\geq2}\sum_{i=3}^{n+1}\sum_{j=0}^{i-2}r_n(i,j)v^iw^jx^n.$$
We seek a formula for $R(x,1,1;p,q)$ and wish to show $$px+R(x,1,1;p,q)=\frac{px}{1-pqx}+U(x;p,1;q).$$

To aid in finding $R(x,v,w)$, define the auxiliary generating function
$$R^+(x,v,w)=\sum_{n\geq4}\sum_{i=4}^n\sum_{j=0}^{i-4}r_n(i,j)v^iw^jx^n,$$
along with
$C(x,v)=\sum_{n\geq2}\sum_{i=0}^{n-1}r_n(i+2,i)v^ix^n$,
$D(x,v)=\sum_{n\geq3}\sum_{i=0}^{n-3}r_n(i+3,i)v^ix^n$, and
$E(x,v)=\sum_{n\geq2}\sum_{j=0}^{n-2}r_n(n+1,j)v^jx^n$. Clearly, $$R(x,v,w)=R^+(x,v,w)+v^2C(x,vw)+v^3D(x,vw)+vE(vx,w).$$

We now rewrite the recurrence relations from Lemma \ref{1324x1423L1} in terms of generating functions.
By \eqref{1324x1423e4}, we have
\begin{align*}
E(x,v)&=\sum_{n\geq2}\sum_{j=0}^{n-2}\sum_{i=1}^{n-1}p^iq^j\binom{n-1-i}{n-j-2}v^jx^n =\sum_{i\geq1}\sum_{n\geq0}\sum_{j=0}^{n-1+i}p^iq^j\binom{n}{j+1-i}v^jx^{n+1+i}\\ &=\sum_{i\geq1}\sum_{j\geq0}\sum_{n\geq j+1}p^iq^{j+i}\binom{n}{j+1}v^{j+i}x^{n+1+i}+\sum_{i\geq1}\sum_{j=0}^{i-1}\sum_{n\geq0}
p^iq^j\binom{n}{j+1-i}v^jx^{n+1+i}\\
&=\sum_{i\geq1}\sum_{j\geq0}\frac{p^iq^{j+i}v^{j+i}x^{j+2+i}}{(1-x)^{j+2}}
+\sum_{i\geq1}\sum_{j=0}^{i-1}\sum_{n\geq0}p^iq^j\binom{n}{j+1-i}v^jx^{n+1+i}\\ &=\sum_{i\geq1}\frac{p^iq^iv^ix^{i+2}}{(1-x)(1-x-qvx)}
+\sum_{i\geq0}\sum_{n\geq0}p^{i+1}q^iv^ix^{n+2+i}\\
&=\frac{pqvx^3}{(1-x)(1-x-qvx)(1-pqvx)}
+\sum_{i\geq0}\frac{p^{i+1}q^iv^ix^{i+2}}{1-x},
\end{align*}
which implies
\begin{align}\label{eq13241423AA1}
E(x,v)&=\frac{px^2}{(1-x-qvx)(1-pqvx)}.
\end{align}
By \eqref{1324x1423e3}, we have
\begin{align}\label{eq13241423AA2}
C(x,v)=\frac{p^2qvx^2}{1-pqvx}.
\end{align}
By \eqref{1324x1423e1}, we have
\begin{align*}
&D(x,v)\\
&=\sum_{n\geq3}\sum_{i=0}^{n-3}pqr_{n-1}(i+2,i-1)v^ix^n
+\sum_{n\geq3}\sum_{i=0}^{n-3}\sum_{j=i+3}^npqr_{n-1}(j,i)v^ix^n\\
&+p\sum_{n\geq3}\sum_{i=0}^{n-3}\sum_{k=i+1}^{n-2}\sum_{a=k+2}^nr_{n-1}(a,k)v^ix^n\\
&=pqvx\sum_{n\geq3}\sum_{i=0}^{n-3}r_n(i+3,i)v^ix^n
+pqx\sum_{n\geq2}\sum_{i=0}^{n-2}\sum_{j=i+3}^{n+1}r_n(j,i)v^ix^n\\
&+px\sum_{n\geq2}\sum_{i=0}^{n-2}\sum_{k=i+1}^{n-1}\sum_{a=k+2}^{n+1}r_n(a,k)v^ix^n\\
&=pqvxD(x,v)+pqxR^+(x,1,v)+pqxE(x,v)+pqxD(x,v) +px\sum_{n\geq2}\sum_{i=0}^{n-2}\sum_{k=i+1}^{n-1}\sum_{a=k+2}^{n+1}r_n(a,k)v^ix^n\\
&=pqvxD(x,v)+pqxR^+(x,1,v)+pqxE(x,v)+pqxD(x,v) +px\sum_{n\geq2}\sum_{k=1}^{n-1}\sum_{a=k+2}^{n+1}\sum_{i=0}^{k-1}r_n(a,k)v^ix^n\\
&=pqvxD(x,v)+pqxR^+(x,1,v)+pqxE(x,v)+pqxD(x,v)
+\frac{px}{1-v}\sum_{n\geq2}\sum_{a=3}^{n+1}\sum_{k=1}^{a-2}r_n(a,k)(1-v^k)x^n,
\end{align*}
which implies
\begin{align}
D(x,v)&=pqvxD(x,v)+pqxR^+(x,1,v)+pqxE(x,v)+pqxD(x,v)\nonumber\\
&+\frac{px}{1-v}(R^+(x,1,1)-R^+(x,1,v))\nonumber\\
&+\frac{px}{1-v}(E(x,1)-E(x,v)+D(x,1)-D(x,v)+C(x,1)-C(x,v)).\label{eq13241423AA3}
\end{align}
By \eqref{1324x1423e2}, we have
\begin{align*}
R^+(x,v,w)&=pqvwx\sum_{n\geq4}\sum_{i=4}^n\sum_{j=0}^{i-4}r_n(i,j)v^iw^jx^n
+pvx\sum_{n\geq3}\sum_{i=3}^n\sum_{j=0}^{i-3}r_n(i,j)v^iw^jx^n\\
&+pqx\sum_{n\geq3}\sum_{i=4}^{n+1}\sum_{j=0}^{i-4}\sum_{a=i}^{n+1}r_n(a,j)v^iw^jx^n\\
&=pqvwxR^+(x,v,w)+pvxR^+(x,v,w)+pv^4xD(x,vw)\\
&+pqx\sum_{n\geq3}\sum_{j=0}^{n-3}\sum_{a=j+4}^{n+1}\sum_{i=j+4}^ar_n(a,j)v^iw^jx^n\\
&=pqvwxR^+(x,v,w)+pvxR^+(x,v,w)+pv^4xD(x,vw)\\
&+\frac{pqvx}{1-v}\sum_{n\geq4}\sum_{a=4}^n\sum_{j=0}^{a-4}r_n(a,j)(v^{j+3}-v^a)w^jx^n\\
&+\frac{pqvx}{1-v}\sum_{n\geq3}\sum_{j=0}^{n-3}r_n(n+1,j)(v^{j+3}-v^{n+1})w^jx^n,
\end{align*}
which implies
\begin{align}\label{eq13241423AA4}
R^+(x,v,w)&=pqvwxR^+(x,v,w)+pvxR^+(x,v,w)+pv^4xD(x,vw)\nonumber\\
&+\frac{pqvx}{1-v}(v^3R^+(x,1,vw)-R^+(x,v,w))+\frac{pqv^2x}{1-v}(v^2E(x,vw)-E(vx,w)).
\end{align}

Using \eqref{eq13241423AA1} and \eqref{eq13241423AA2}, we can write \eqref{eq13241423AA3} and \eqref{eq13241423AA4} as
\begin{align}
&\left(1-pqvx-pqx+\frac{px}{1-v}\right)D(x,v)=(pqx-\frac{px}{1-v})R^+(x,1,v)+\frac{px}{1-v}R^+(x,1,1)\nonumber\\
&\qquad+\frac{px}{1-v}D(x,1)+\frac{(qx-1)(pqvx+pqx+px-p-1)p^2qx^3}{(pqx-1)(qx+x-1)(pqvx-1)(qvx+x-1)}.\label{eq13241423AM2}\\
&\left(1-pqvwx-pvx+\frac{pqvx}{1-v}\right)R^+(x,v,w)\nonumber=pv^4xD(x,vw)+\frac{pqv^4x}{1-v}R^+(x,1,vw)\nonumber\\
&\qquad-\frac{p^2qv^4x^4}{(pqvwx-1)(qvwx+x-1)(qvwx+vx-1)},\label{eq13241423AM1}
\end{align}
Replacing $w$ with $w/v$ in \eqref{eq13241423AM1} then gives
\begin{align*}
&\left(1-pqwx-pvx+\frac{pqvx}{1-v}\right)R^+(x,v,w/v)\\
&\qquad=pv^4xD(x,w)+\frac{pqv^4x}{1-v}R^+(x,1,w)-\frac{p^2qv^4x^4}{(pqwx-1)(qwx+x-1)(qwx+vx-1)}.
\end{align*}
Upon applying the kernel method and taking
$$v=v_0(x,w)=\frac{1-pqwx-pqx+px-\sqrt{(1-pqwx-pqx+px)^2-4px(1-pqwx)}}{2px},$$
one gets
\begin{align}\label{eq13241423AN1}
R^+(x,1,w)&=\frac{px^3(v_0(x,w)-1)}{(1-pqwx)(1-x-qwx)(1-qwx-xv_0(x,w))}+\frac{v_0(x,w)-1}{q}D(x,w).
\end{align}
Substituting \eqref{eq13241423AN1} into \eqref{eq13241423AM2} yields
\begin{align}
&\left(q(1-pqvx-pqx)(1-v)+pqx-px(q(1-v)-1)(v_0(x,v)-1)\right)D(x,v)\nonumber\\
&=px(v_0(x,1)+q-1)D(x,1)+\frac{(q(1-v)-1)(v_0(x,v)-1)p^2qx^4}{(1-pqvx)(1-x-qvx)(1-qvx-xv_0(x,v))}\label{eq13241423AN2}\\
&+\frac{(v_0(x,1)-1)p^2qx^4}{(1-pqx)(1-x-qx)(1-qx-xv_0(x,1))}\nonumber\\
&-\frac{(1-qx)(1-v)(pqvx+pqx+px-p-1)p^2q^2x^3}{(1-pqx)(1-x-qx)(1-pqvx)(1-x-qvx)}.\nonumber
\end{align}
Note that the solution of the equation
$$q(1-pqx-pqvx)(1-v)+pqx-px(q(1-v)-1)(v_0(x,v)-1)=0,$$
is given by
$$v=v_1(x)=\frac{1+2pqx-px-\sqrt{(2q-1)^2p^2x^2-2(2q+1)px+1}}{4pqx}.$$
Thus taking $v=v_1(x)$ in \eqref{eq13241423AN2} yields
\begin{align}
D(x,1)&=\frac{(q(v_1(x)-1)+1)(v_0(x,v_1(x))-1)pqx^3}{(1-pqxv_1(x))(1-x-qxv_1(x))(1-qxv_1(x)-xv_0(x,v_1(x)))(v_0(x,1)-1+q)}\nonumber\\
&-\frac{(v_0(x,1)-1)pqx^3}{(1-pqx)(1-x-qx)(1-qx-xv_0(x,1))(v_0(x,1)-1+q)}\nonumber\\
&-\frac{(1-qx)(v_1(x)-1)(pqxv_1(x)+pqx+px-p-1)pq^2x^2}{(1-pqx)(1-x-qx)(1-pqxv_1(x))(1-x-qxv_1(x))(v_0(x,1)-1+q)}.\label{eq13241423DD1}
\end{align}

Using  \eqref{eq13241423AN2} again, together with the expression for $D(x,1)$, we obtain explicitly the generating function $D(x,v)$:
\begin{align}
K\cdot D(x,v)&=p^3q^2x^5v_0(x,v)v_0(x,v_1(x))(v_1(x)-v)
-(p+1)p^2q^2x^4v_0(x,v)v_1(x)\nonumber\\
&-(pq^2vx-pqvx-pqv-q+1)p^2qx^4v_0(x,v)
+p^3q^2(qx-x-1)x^4v_0(x,v_1(x))v_1(x)\nonumber\\
&-(pq^2v^2x-pqv-qv+q-1)p^2qx^4v_0(x,v_1(x))
-(p+1)p^2q^2(qx-x-1)x^3v_1(x)\nonumber\\
&-p^2q^2v(qx-x-1)(pqvx-p-1)x^3,\label{eq13241423ANM2}
\end{align}
where $K=(pqvx-1)(qvx+xv_0(x,v)-1)(pq^2(v^2-1)x+px(qv-q+1)v_0(x,v)-pqvx+2pqx-px-qv+q)
(pqxv_1(x)-1)(qxv_1(x)+xv_0(x,v_1(x))-1)$. Hence, by \eqref{eq13241423AN1}, we obtain a formula for $R^+(x,1,v)$ and then by \eqref{eq13241423AM1}, a formula for $R^+(x,v,w)$:

\noindent $L\cdot R^+(x,v,w)=-p^3q^2v^4x^5(pq^2(v^2w^2-vw-1)x^2+(pv^2+2pq-p)x^2+(p-1)qvwx+(pq-p+q)x-p)v_0(x,vw)v_0(x,v_1(x))v_1(x)
-v^4qp^2x^5(p^2q^3v^3w^3x^2-2p^2q^2v^2w^2x-p^2qv^3wx^2-2pq^2v^2w^2x+pq^2vwx+p^2qvw+pq^2x+pqvw-2pqx+qvw-pq+px+p-q)v_0(x,vw)v_0(x,v_1(x))
+p^2q^2v^4x^4(p+1)(pq^2v^2w^2x^2-pq^2vwx^2-pq^2x^2+pqvwx+pv^2x^2+2pqx^2-qvwx+pqx-px^2-px+qx-p)v_0(x,vw)v_1(x)
-v^4qp^2x^4(p^2q^4v^3w^3x^3-p^2q^3v^3w^3x^3-p^2q^3v^3w^3x^2-2p^2q^3v^2w^2x^2-p^2q^2v^3wx^3+2p^2q^2v^2w^2x^2+p^2qv^3wx^3-pq^3v^2w^2x^2+2p^2q^2v^2w^2x+p^2qv^3wx^2+pq^2v^2w^2x^2+2pq^2v^2w^2x+p^2q^2vwx-p^2qvwx+pq^2vwx+pqv^2x^2-p^2qvw-2pqvwx-pv^2x^2-pq^2x-pqvw+2pqx-qvw-px+q)v_0(x,vw)
+v^5q^2p^3x^5(pq^2v^2w^2x^2-pq^2vwx^2-pq^2x^2+pqvwx-pqvx^2+2pqx^2+pvx^2-qvwx+pqx+pvx-px^2-px+qx-p)v_0(x,v_1(x))v_1(x)
+v^5qp^2x^5(p^2q^3v^3w^3x^2+p^2q^2v^3w^2x^2-2p^2q^2v^2w^2x-2pq^2v^2w^2x-p^2qv^2wx+pq^2vwx-pqv^2wx+p^2qvw+pq^2x+pqvw+pqvx-2pqx-pvx+qvw-pq+px+p-q)v_0(x,v_1(x))
-v^5q^2p^2x^4(p+1)(pq^2v^2w^2x^2-pq^2vwx^2-pq^2x^2+pqvwx-pqvx^2+2pqx^2+pvx^2-qvwx+pqx+pvx-px^2-px+qx-p)v_1(x)
+v^5qp^2x^4(p^2q^4v^3w^3x^3-p^2q^3v^3w^3x^3-p^2q^3v^3w^3x^2+p^2q^3v^3w^2x^3-p^2q^2v^3w^2x^3-2p^2q^3v^2w^2x^2-p^2q^2v^3w^2x^2+2p^2q^2v^2w^2x^2-pq^3v^2w^2x^2+2p^2q^2v^2w^2x-p^2q^2v^2wx^2+pq^2v^2w^2x^2+p^2qv^2wx^2+2pq^2v^2w^2x-pq^2v^2wx^2+p^2q^2vwx+p^2qv^2wx+pqv^2wx^2-p^2qvwx+pq^2vwx+pqv^2wx-p^2qvw-2pqvwx-pq^2x-pqvw+2pqx-qvw-px+q)$,

\noindent where $L=(pqv^2wx-pqvwx+pqvx+pv^2x-pvx-v+1)(qvwx+vx-1)(pqvwx-1)
(qvwx+xv_0(x,vw)-1)(pq^2v^2w^2x+pqvwxv_0(x,vw)-pq^2x-pqvwx-pqxv_0(x,vw)+2pqx+pxv_0(x,vw)-px-qvw+q)(pqxv_1(x)-1)(qxv_1(x)+xv_0(x,v_1(x))-1)$.

Therefore, we can state the following result.

\begin{theorem}\label{Rgenfun}
We have
\begin{align*}
R(x,v,w)=R^+(x,v,w)+v^2C(x,vw)+v^3D(x,vw)+vE(vx,w),
\end{align*}
where $R^+(x,v,w)$ is as above and $C(x,v)$, $D(x,v)$ and $E(x,v)$ are given in \eqref{eq13241423AA2}, \eqref{eq13241423ANM2} and \eqref{eq13241423AA1}, respectively.
\end{theorem}

As a corollary to this result, one can show (with the aid of programming) that
\begin{align*}
vx+R(x,1,1;v,q)&=\frac{vx}{1-vqx}+\frac{vx(vqx-v-x)t(xv)}{2(vq^2x^2-vqx^2-qx+vx-v-x+1)(vqx-vx-1)}\\
&+\frac{(x+(qx^2+qx+3x^2-2x-1)v)vx}{2(vqx-vx-1)(vqx-1)(vq^2x^2-vqx^2-qx+vx-v-x+1)}\\
&-\frac{(2q^2x^2+3q^2x-qx^2-qx-3q+2x-1+(1-2q)(qx-1)vqx)v^3x^2}{2(vqx-vx-1)(vqx-1)(vq^2x^2-vqx^2-qx+vx-v-x+1)}\\
&=\frac{vx}{1-vqx}+U(x;v,1;q),
\end{align*}
which establishes the desired equivalence of distributions in the case of $(1324,1423)$.

\section{The case $(1342,1423)$.}

To enumerate members of this class according to the number of descents, we consider a more general version rather than a refinement of the problem at hand.  Given $n \geq 1$, $1 \leq m \leq n$ and $1 \leq i \leq n$, let $\mathcal{D}_n(i,m)$ denote the subset of $\mathcal{S}_{n,i}(1342,1423)$ consisting of those members in which the letters $n,n-1,\ldots,n-m+1$ form a decreasing subsequence.  Let $\mathcal{D}_n(m)=\cup_{i=1}^n\mathcal{D}_n(i,m)$ for $n \geq 1$ and $m\in [n]$.  Define the distributions $d_n(i,m)=d_n(i,m;q)$ and $D_n(m)=D_n(m;q)$ by
$$d_n(i,m)=\sum_{\pi\in\mathcal{D}_n(i,m)}q^{\text{desc}(\pi)}$$
and
$$D_n(m)=\sum_{\pi\in\mathcal{D}_n(m)}q^{\text{desc}(\pi)}.$$
Note that $D_n(m)=\sum_{i=1}^nd_n(i,m)$, by the definitions.

To aid in finding a recurrence for $d_n(i,m)$, we consider the subset of $\mathcal{D}_{n}(m)$ comprising its indecomposable members.  Let $\mathcal{E}_n(m)$ denote the subset of $\mathcal{D}_n(m)$ consisting of those $\sigma$ which cannot be decomposed as $\sigma=\sigma'\sigma''$ where $\sigma'$ contains $[n-m+1,n]$ and $\sigma''$ is a permutation of $[a]$ for some $a \geq 1$. Define $e_n(m)=e_n(m;q)$ by
  $$e_n(m)=\sum_{\pi\in\mathcal{E}_n(m)}q^{\text{desc}(\pi)}.$$

Note that $d_n(i,m)$ may assume nonzero values only when $n \geq 1$ and $i,m \in [n]$.  Further, we have $d_n(i,m)=0$ if $n-m+1 \leq i \leq n-1$, since elements of $[n-m+1,n]$ must decrease, with $d_n(n,n)=q^{n-1}=e_n(n)$ and $d_n(i,n)=0$ if $i<n$.  One may verify directly the following values of $d_n(i,m)$ and $e_n(m)$ when $n=3$: $d_3(1,1)=q+1$, $d_3(1,2)=q$, $d_3(1,3)=0$; $d_3(2,1)=2q$, $d_3(2,2)=d_3(2,3)=0$; $d_3(3,1)=d_3(3,2)=q+q^2$, $d_3(3,3)=q^2$, with $e_3(1)=1+2q$, $e_3(2)=2q$ and $e_3(3)=q^2$.  It can be shown in general that $d_n(1,n-1)=q^{n-2}$ and $d_n(n,n-1)=(n-2)q^{n-2}+q^{n-1}$ for $n \geq 2$, with $e_n(n-1)=(n-1)q^{n-2}$.

The $d_n(i,m)$ and $e_n(m)$ satisfy the following system of intertwined recurrences.

\begin{lemma}
If $n \geq 1$, then
\begin{equation}\label{a13421423rec1}
  e_n(m)=D_n(m)-q\sum_{a=1}^{n-m}e_{n-a}(m)D_a(1), \quad 1 \leq m \leq n.
  \end{equation}
If $n \geq 2$, then
 \begin{equation}\label{a13421423rec2}
  d_n(n,m)=qD_{n-1}(m-1), \quad 2 \leq m \leq n,
  \end{equation}
  with $d_n(n,1)=qD_{n-1}(1)$.
   If $n \geq 3$, then
  \begin{equation}\label{a13421423rec3}
  d_n(n-m,m)=qD_{n-2}(m-1)+q\sum_{j=1}^{n-m-1}d_{n-1}(j,m), \quad 2 \leq m \leq n-1,
  \end{equation}
  with $d_n(n-1,1)=qD_{n-2}(1)+q\sum_{j=1}^{n-2}d_{n-1}(j,1)$.
  If $3 \leq m+2 \leq n$ and $1 \leq i \leq n-m-1$, then
  \begin{align}
  d_n(i,m)&=d_{n-1}(i,m)+qD_{n-2}(n-i-1)+q\sum_{j=1}^{i-1}d_{n-1}(j,m)\notag\\
  &\quad+q\sum_{j=i+2}^{n-m}\sum_{a=0}^{i-1}e_{j-a-2}(j-i-1)d_{n-j+a+1}(a+1,m), \label{a13421423rec4}
  \end{align}
    with $d_n(i,n)=\delta_{i,n}\cdot q^{n-1}$ for $1 \leq i \leq n$, $d_n(i,m)=0$ for $n-m+1 \leq i \leq n-1$ and $d_2(1,1)=1$, $d_2(2,1)=q$.
  \end{lemma}
\begin{proof}
The conditions above stated last may be verified using the definitions.  To show \eqref{a13421423rec1}, note that $\pi \in \mathcal{D}_n(m)$ is either indecomposable or of the form $\pi=\pi'\pi''$, where $\pi'$ contains $[n-m+1,n]$ and $\pi''$ is a permutation of $[a]$ for some $1 \leq a \leq n-m$ with $a$ maximal.  The range of $a$ implies that both $\pi'$ and $\pi''$ are nonempty and hence a descent occurs between the last letter of $\pi'$ and the first of $\pi''$.  Considering all possible $a$ then yields \eqref{a13421423rec1}.  Formula \eqref{a13421423rec2} follows from removing $n$ from $\pi \in \mathcal{D}_{n}(n,m)$, which results in a member of $\mathcal{D}_{n-1}(m-1)$ if $m \geq 2$ or of $\mathcal{D}_{n-1}(1)$ if $m=1$.  To show \eqref{a13421423rec3}, first note that $\pi \in \mathcal{D}_n(n-m,m)$ for $n \geq 3$ must have second letter $j<n-m$ or $j=n$, for $j \in [n-m+1,n-1]$ is disallowed as the elements of $[n-m+1,n]$ are to decrease.  If $j<n-m$, then removal of $n-m$ from $\pi$ results in a member of $\mathcal{D}_{n-1}(j,m)$, whereas if $j=n$, then removal of both $n-m$ and $n$ yields a member of $\mathcal{D}_{n-2}(m-1)$ if $m\geq 2$ or of $\mathcal{D}_{n-2}(1)$ if $m=1$.  Since a removed letter is part of a descent in either case, formula \eqref{a13421423rec3} follows from considering all possible values of $j$.

To show \eqref{a13421423rec4}, suppose $\pi \in \mathcal{D}_n(i,m)$ where $n \geq m+2$ and $1 \leq i \leq n-m-1$.  If $\pi$ starts $i,j$ for some $j<i$, or starts $i(i+1)$, then there are clearly $q\sum_{j=1}^{i-1}d_{n-1}(j,m)$ and $d_{n-1}(i,m)$ possibilities, respectively.  If $\pi$ starts with $i,n$, then both letters may be deleted, wherein it is required that the elements of $[i+1,n-1]$ decrease.  Since $i<n-m$, we have that $[n-m+1,n-1]$ is contained in $[i+1,n-1]$ and thus there are $qD_{n-2}(n-i-1)$ possibilities for the remaining letters of $\pi$, where the factor of $q$ accounts for the descent of $\pi$ arising due to $n$.  This completes the proof of \eqref{a13421423rec4} in the case when $i=n-m-1$, so assume $i<n-m-1$.  In this case, it is possible for the second letter to be greater than $i$, but equal to neither $i+1$ nor $n$.

So suppose $\pi \in \mathcal{D}_n(i,m)$ where $1 \leq i \leq n-m-2$ has second letter $j \in [i+2,n-m]$.  Before enumerating this case, we establish a certain decomposition of $\pi$.  Note first that the elements of $[i+1,j-1]$ must occur (in decreasing order) prior to any elements of $[j+1,n]$, for otherwise a $1423$ or $1342$ would occur.  Let $\mathcal{C}$ denote the section of $\pi$ starting with $i+1$ and ending with the predecessor of the leftmost element of $[j+1,n]$.  We claim that for some $c \in \mathcal{C}$, it must be the case that $\pi$ can be decomposed as $\pi=\alpha c \beta$, where the section $\alpha c$ comprises all elements in an interval of the form $[d,j]$ for some $d \in [i]$.  Note that if $\pi$ starts $ij(j-1)(j-2)\cdots(i+1)$, then one may take $c=i+1$ and $d=i$, so assume that this is not the case. Let $\pi=\ell_1\rho_1\cdots \ell_r\rho_r$, where $\ell_1,\ldots,\ell_r$ denotes the complete set of left-right minima of $\pi$ and the $\rho_i$ are possibly empty.  By a \emph{unit} of $\pi$, we mean a section $\ell_i\rho_i$ for some $1 \leq i \leq r$.  By the assumption above on $\pi$, it must be the case that $i+1$ belongs to some unit $\mathcal{U}$ having first letter $u$ where $u \in [i-1]$.

We consider cases based on whether or not a member of $[j+1,n]$ lies in $\mathcal{U}$.  If so, then any elements of $[u+1,i-1]$ must occur prior to $z$, where $z$ denotes the leftmost letter in $[j+1,n]$, for otherwise $u(i+1)zt$ for some $t\in [u+1,i-1]$ would be a $1342$.  Then the predecessor of $z$ would be the desired letter $c$ in this case, with $\alpha c$ comprising the interval $[u,j]$.  So assume all elements of $[j+1,n]$ occur strictly to the right of $u'$, where $u'$ denotes the last letter of $\mathcal{U}$.  If $u'$ fits the criterion for $c$, then we are done.  Otherwise, $u'$ is followed by a left-right minima $v<u$, with at least one element of $[u+1,i-1]$ occurring to the right of $v$.  Let $S=\{s_1,s_2,\ldots,s_r\}$ denote the set of such elements in $[u+1,i-1]$, which must occur in decreasing order, for otherwise there would be an occurrence of $1423$ starting with $u(i+1)$.

The problem then essentially starts anew as the elements of $S$ must decrease amongst letters in $[v]$ and occur prior to any letters in $[j+1,n]$ (for otherwise, a 1342 would occur starting with $u(i+1)$).  Note that the roles of $i$ and $j$ in the original problem are now roughly played by $u$ and $i+1$.  Upon relabeling the $s_i$ as members of $[u,u+r-1]$, the new subproblem is equivalent to considering permutations of $[u+r+n-j-1]$ starting with $v \in [u-1]$ in which the elements of $[u,u+r-1]$ must decrease and occur prior to any elements of $[u+r,u+r+n-j-1]$.  Since there are strictly fewer letters in the smaller subproblem, one can proceed inductively.  Let $\pi'$ denote the section of $\pi$ to the right of and including $v$ and assume letters in $\pi'$ are relabeled as elements of $[u+r+n-j-1]$.  By the induction hypothesis, there exists some $c'$ such that $\pi'$ may be written as $\pi'=\alpha'c'\beta'$, where $\alpha'c'$ comprises an interval of the form $[d',u+r-1]$ for some $d'\in[v]$.  Then within $\pi$, one may verify that $c'$ furnishes the desired element $c$, which completes the induction.

Consider the leftmost $c$ for which $\pi$ may be decomposed as $\pi=\alpha c \beta$ such that $\alpha c$ is of the form $[a+1,j]$ for some $0 \leq a \leq i-1$.  Then $\alpha c-\{i,j\}$ has length $j-a-2$, with $[i+1,j-1]$ decreasing so that it is enumerated by $e_{j-a-2}(j-i-1)$.  Note that $\beta$ comprises the letters in $[a]\cup[j+1,n]$ and effectively starts with $a+1$ since $j$ occurs to the left of all the letters in $[j+1,n]$.  As there are no other restrictions on $j\beta$, it is enumerated by $d_{n-j+a+1}(a+1,m)$ since the elements of $[n-m+1,n]$ are required to decrease.  Note there is an additional descent that may be attributed to the successor of $j$ in $\pi$, which must belong to $[i-1]\cup\{j-1\}$.  Thus, considering all possible $a$ and $j$ yields $q\sum_{j=i+2}^{n-m}\sum_{a=0}^{i-1}e_{j-a-2}(j-i-1)d_{n-j+a+1}(a+1,m)$ additional members of $\mathcal{D}_n(i,m)$ and combining this with the previous cases yields \eqref{a13421423rec4} and completes the proof.
\end{proof}

We define the generating functions
$$D(x,v,w)=\sum_{n\geq1}\sum_{m=1}^n\sum_{i=1}^nd_n(i,m)v^iw^{m-1}x^n$$ and
$$E(x,v)=\sum_{n\geq1}\sum_{m=1}^ne_n(m)v^{m-1}x^n.$$
We spilt $D(x,v,w)$ into three further generating functions as follows:
$$D_1(x,v,w)=\sum_{n\geq3}\sum_{m=1}^{n-2}\sum_{i=1}^{n-m-1}d_n(i,m)v^iw^{m-1}x^n,$$
$$D_2(x,v,w)=\sum_{n\geq2}\sum_{m=1}^{n-1}d_n(n-m,m)v^{n-m}w^{m-1}x^n$$
and
$$D_3(x,v,w)=\sum_{n\geq1}\sum_{m=1}^{n}d_n(n,m)v^nw^{m-1}x^n.$$  Note that
$$D(x,v,w)=D_1(x,v,w)+D_2(x,v,w)+D_3(x,v,w).$$

By \eqref{a13421423rec1}, we have
\begin{align*}
E(x,v)&=\sum_{n\geq1}\sum_{m=1}^ne_n(m)v^{m-1}x^n\\
&=\sum_{n\geq1}\sum_{m=1}^nD_n(m)v^{m-1}x^n-q
\sum_{n\geq1}\sum_{m=1}^n\sum_{a=1}^{n-m}e_{n-a}(m)D_a(1)v^{m-1}x^n\\
&=D(x,1,v)-q\sum_{n\geq2}\sum_{a=1}^{n-1}\sum_{m=1}^{n-a}e_{n-a}(m)D_a(1)v^{m-1}x^n\\
&=D(x,1,v)-qE(x,v)D(x,1,0).
\end{align*}
By \eqref{a13421423rec2}, we have
\begin{align*}
D_3(x,v,w)&=\sum_{n\geq1}d_n(n,1)v^nx^n+\sum_{n\geq2}\sum_{m=2}^nd_n(n,m)v^nw^{m-1}x^n\\
&=qvxD(vx,1,0)+vx+qvwx\sum_{n\geq1}\sum_{m=1}^{n}\sum_{i=1}^nd_n(i,m)v^nw^{m-1}x^n\\
&=qvxD(vx,1,0)+vx+qvwxD(vx,1,w).
\end{align*}

By \eqref{a13421423rec3}, we have
\begin{align*}
D_2(x,v,w)&=\sum_{n\geq2}d_n(n-1,1)v^{n-1}x^n+\sum_{n\geq3}\sum_{m=2}^{n-1}
d_n(n-m,m)v^{n-m}w^{m-1}x^n\\
&=vx^2+\sum_{n\geq3}\left(q\sum_{i=1}^{n-2}d_{n-2}(i,1)
+q\sum_{i=1}^{n-2}d_{n-1}(i,1)\right)v^{n-1}x^n\\
&+\sum_{n\geq3}\sum_{m=2}^{n-1}\left(q\sum_{i=1}^{n-2}d_{n-2}(i,m-1)
+q\sum_{j=1}^{n-1-m}d_{n-1}(j,m)\right)v^{n-m}w^{m-1}x^n\\
&=vx^2+qvx^2\sum_{n\geq1}\sum_{i=1}^nd_n(i,1)v^nx^n
+qx\sum_{n\geq1}\sum_{i=1}^nd_n(i,1)v^nx^n-qx\sum_{n\geq1}d_n(n,1)v^nx^n\\
&+\sum_{n\geq3}\sum_{m=2}^{n-1}\left(q\sum_{i=1}^{n-2}d_{n-2}(i,m-1)
+q\sum_{j=1}^{n-1-m}d_{n-1}(j,m)\right)v^{n-m}w^{m-1}x^n\\
&=(1-q)vx^2+qx(1+(1-q)vx)D(vx,1,0)\\
&+\sum_{n\geq3}\sum_{m=2}^{n-1}\left(q\sum_{i=1}^{n-2}d_{n-2}(i,m-1)
+q\sum_{j=1}^{n-1-m}d_{n-1}(j,m)\right)v^{n-m}w^{m-1}x^n\\
&=(1-q)vx^2+qx(1+(1-q)vx)D(vx,1,0)\\
&+qvwx^2\sum_{n\geq1}\sum_{m=1}^n\sum_{i=1}^nd_n(i,m)v^{n-m}w^{m-1}x^n+q\sum_{n\geq3}\sum_{m=2}^{n-1}\sum_{j=1}^{n-1-m}d_{n-1}(j,m)v^{n-m}w^{m-1}x^n\\
&=(1-q)vx^2+qx(1+(1-q)vx)D(vx,1,0)+qwx^2D(vx,1,w/v)\\
&+q\sum_{n\geq3}\sum_{m=2}^{n-1}\sum_{j=1}^{n-2}d_{n-1}(j,m)v^{n-m}w^{m-1}x^n\\
&=(1-q)vx^2+qx(1+(1-q)vx)D(vx,1,0)+qwx^2D(vx,1,w/v)\\
&+qvx\sum_{n\geq2}\sum_{m=1}^n\sum_{j=1}^{n-1}d_n(j,m)v^{n-m}w^{m-1}x^n-qvx\sum_{n\geq2}\sum_{j=1}^{n-1}d_n(j,1)v^{n-1}x^n\\
&=(1-q)vx^2+qx(1+(1-q)vx)D(vx,1,0)+qwx^2D(vx,1,w/v)\\
&+qx\sum_{n\geq2}\sum_{m=1}^n\sum_{j=1}^{n-m}d_n(j,m)(w/v)^{m-1}(vx)^n-qx\sum_{n\geq1}\sum_{j=1}^nd_n(j,1)(vx)^n\\
&+qx\sum_{n\geq1}d_n(n,1)(vx)^n\\
&=(1-q)vx^2+qx(1+(1-q)vx)D(vx,1,0)+qwx^2D(vx,1,w/v)\\
&+qx(D_1(vx,1,w/v)+vD_2(x,v,w))-qxD(vx,1,0)+qx(qvxD(vx,1,0)+vx)\\
&=vx^2+qvx^2D(vx,1,0)+qwx^2D(vx,1,w/v)+qx(D_1(vx,1,w/v)+vD_2(x,v,w)).
\end{align*}

By \eqref{a13421423rec4}, we have
\begin{align*}
D_1(x,v,w)&=\sum_{n\geq3}\sum_{m=1}^{n-2}\sum_{i=1}^{n-1-m}d_n(i,m)v^iw^{m-1}x^n\\
&=\sum_{n\geq3}\sum_{m=1}^{n-2}\sum_{i=1}^{n-1-m}\left(
d_{n-1}(i,m)+q\sum_{j=1}^{n-2}d_{n-2}(j,n-i-1)+q\sum_{j=1}^{i-1}d_{n-1}(j,m)\right)v^iw^{m-1}x^n\\
&+\sum_{n\geq3}\sum_{m=1}^{n-2}\sum_{i=1}^{n-1-m}\left(
q\sum_{j=i+2}^{n-m}\sum_{a=0}^{i-1}e_{j-2-a}(j-1-i)d_{n-j+a+1}(a+1,m)\right)v^iw^{m-1}x^n\\
&=x(D_1(x,v,w)+D_2(x,v,w)) +\frac{qvx^2}{1-w}\sum_{n\geq1}\sum_{i=1}^n\sum_{a=1}^nd_n(a,i)v^{n-i}(1-w^i)x^n\\
&+\frac{qvx}{1-v}\sum_{n\geq3}\sum_{m=1}^{n-2}\sum_{j=1}^{n-1-m}
d_n(j,m)(v^j-v^{n-m})w^{m-1}x^n\\
&+qx^2\sum_{m\geq1}\sum_{n\geq1}\sum_{j=1}^n\sum_{a=0}^{j-1}\sum_{i=a+1}^j e_{j-a}(j+1-i)d_{n+m-j+a+1}(a+1,m)v^iw^{m-1}x^{n+m}\\
&=x(D_1(x,v,w)+D_2(x,v,w))+\frac{qx^2}{1-w}(D(vx,1,1/v)-wD(vx,1,w/v))\\
&+\frac{qx}{1-v}(vD_1(x,v,w)-D_1(vx,1,w/v))\\
&+qx^2\sum_{m\geq1}\sum_{a\geq1}\sum_{j\geq1}\sum_{n\geq j+m}\sum_{i=1}^ae_a(i)v^{-i}d_n(j,m)v^{j+a}w^{m-1}x^{n+a-1}\\
&=x(D_1(x,v,w)+D_2(x,v,w))+\frac{qx^2}{1-w}(D(vx,1,1/v)-wD(vx,1,w/v))\\
&+\frac{qx}{1-v}(vD_1(x,v,w)-D_1(vx,1,w/v))\\
&+qx\sum_{m\geq1}\sum_{a\geq1}\sum_{j\geq1}\sum_{n\geq j+m+a}\sum_{i=1}^ae_a(i)v^{-i}d_{n-a}(j,m)v^{j+a}w^{m-1}x^n\\
&=x(D_1(x,v,w)+D_2(x,v,w))+\frac{qx^2}{1-w}(D(vx,1,1/v)-wD(vx,1,w/v))\\
&+\frac{qx}{1-v}(vD_1(x,v,w)-D_1(vx,1,w/v))\\
&+qx\sum_{n\geq3}\sum_{a=1}^{n-2}\left(\sum_{i=1}^ae_a(i)v^{-i}
\sum_{m=1}^{n-1-a}\sum_{j=1}^{n-a-m}d_{n-a}(j,m)v^{j+a}w^{m-1}\right)x^n\\
&=x(D_1(x,v,w)+D_2(x,v,w))+\frac{qx^2}{1-w}(D(vx,1,1/v)-wD(vx,1,w/v))\\
&+\frac{qx}{1-v}(vD_1(x,v,w)-D_1(vx,1,w/v))+\frac{qx}{v}E(vx,1/v)(D_1(x,v,w)+D_2(x,v,w)).
\end{align*}

Hence, there is the following system of functional equations.

\begin{lemma}\label{lem13421423a}
We have $D(x,v,w)=D_1(x,v,w)+D_2(x,v,w)+D_3(x,v,w)$, where
\begin{align*}
E(x,v)&=\frac{D(x,1,v)}{1+qD(x,1,0)},\\
D_1(x,v,w)&=x(D_1(x,v,w)+D_2(x,v,w))+\frac{qx^2}{1-w}(D(vx,1,1/v)-wD(vx,1,w/v))\\
&+\frac{qx}{1-v}(vD_1(x,v,w)-D_1(vx,1,w/v))+\frac{qx}{v}E(vx,1/v)(D_1(x,v,w)+D_2(x,v,w)),\\
D_2(x,v,w)&=vx^2+qvx^2D(vx,1,0)+qwx^2D(vx,1,w/v)+qx(D_1(vx,1,w/v)+vD_2(x,v,w)),\\
D_3(x,v,w)&=qvxD(vx,1,0)+vx+qvwxD(vx,1,w).
\end{align*}
\end{lemma}

In order to solve these equations, we assume
\begin{align*}
{\bf P_1}:&\,\,D(x,v,0)=vx+R(x,v),\\
{\bf P_2}:&\,\,\frac{D_1(x,1,w)}{D_2(x,1,w)}=\frac{1-x-t(x)}{2qx(-qx+x+1)}-1,
\end{align*}
where
\begin{align*}
R(x,v)&=\frac{vx(qvx-v-x)t(vx)}{2((q-1)vx-1)(q(q-1)vx^2-qx+(v-1)(x-1))}\\
&-\frac{vx(2q(q-1)^2v^2x^3-(2q^2-3q+2)v^2x^2+(3-2q^2)vx^2+(3q-2)vx+v^2x-v+x)}{2((q-1)vx-1)(q(q-1)vx^2-qx+(v-1)(x-1))}
\end{align*}
and
$$t(x)=\sqrt{(1-2q)^2x^2-2x(1+2q)+1}.$$
First, we will determine below a potential solution to the system of equations in Lemma \ref{lem13421423a} with the aid of the assumptions ${\bf P_1}$ and ${\bf P_2}$.  One may then verify that this potential solution indeed satisfies the equations in Lemma \ref{lem13421423a} (as well as ${\bf P_1}$ and ${\bf P_2}$), and hence it is the desired solution.  Note that $D(x,v,0)$ is the actual generating function that is sought in this case since it enumerates all members of $\mathcal{D}_n(1)$.

By Lemma \ref{lem13421423a}, we have
\begin{align*}
D(x,1,w)&=D_1(x,1,w)+D_2(x,1,w)+D_3(x,1,w),\\
D_2(x,1,w)&=x^2+qx^2D(x,1,0)+qwx^2D(x,1,w)+qx(D_1(x,1,w)+D_2(x,1,w)),\\
D_3(x,1,w)&=qxD(x,1,0)+x+qwxD(x,1,w).
\end{align*}
Solving this system under ${\bf P_1}$ and ${\bf P_2}$, we get explicit formulas for $D_j(x,1,w)$ for $j=1,2,3$ and $D(x,1,w)$. Hence, by Lemma \ref{lem13421423a}, we have an explicit formula for $E(x,v)$, which is given by  $E(x,v)=\frac{D(x,1,v)}{1+qD(x,1,0)}$.

Therefore, again by Lemma \ref{lem13421423a}, we have
\begin{align*}
&D(x,v,w)=D_1(x,v,w)+D_2(x,v,w)+D_3(x,v,w),\\
&D_1(x,v,w)\\
&=\frac{\frac{qx^2}{1-w}(D(vx,1,1/v)-wD(vx,1,w/v))
-\frac{qx}{1-v}D_1(vx,1,w/v)+x(1+\frac{q}{v}E(vx,1/v))D_2(x,v,w)}{1-x-\frac{qx}{v}E(vx,1/v)-\frac{qvx}{1-v}},\\
&D_2(x,v,w)=\frac{vx^2+qvx^2D(vx,1,0)+qwx^2D(vx,1,w/v)+qxD_1(vx,1,w/v)}{1-qvx},\\
&D_3(x,v,w)=qxvD(vx,1,0)+vx+qvwxD(vx,1,w),
\end{align*}
which leads to explicit formulas for the $D_j(x,v,w)$ and $D(x,v,w)$.

\begin{theorem}\label{th13421423a}
We have $D(x,v,w)=D_1(x,v,w)+D_2(x,v,w)+D_3(x,v,w)$, where
\begin{align*}
D_1(x,v,w)&=\frac{4vx^3(1+q+q(1-q)vx)}{b_1+b_2t(vx)},\\
b_1&=(2q+1)(q+1)wx^2-2(w+1)qx-2x+2
-\bigl((6q^3-q^2-4q+1)wx^2\\
&-(6q^2-3q-1)wx+(2-6q^2)x+6q\bigr)vx\\
&+(2q-1)(q-1)(qx-1)((2q-1)wx-2)v^2x^2,\\
b_2&=(2q-1)(q+1)wx^2-2(w+1)qx-2x+2-(q-1)(qx-1)((2q-1)wx-2)vx,\\
D_2(x,v,w)&=\frac{4vx^2}{2-(4qv+2qw-2v+w)x+vw(2q-1)^2x^2+(2+wx-2qwx)t(vx)},\\
D_3(x,v,w)&=\frac{8vx}{6-w-2v(2qw+2q-1)x+v^2w(2q-1)^2x^2+(w+2-(2q-1)vwx)t(vx)}.
\end{align*}
Moreover,
\begin{align*}
E(x,v)&=\frac{1-x-2qvx+2qx-t(x)}{2q(qv^2x-2qvx+vx-v+2)}.
\end{align*}
\end{theorem}
\begin{proof}
By direct calculation (with the aid of mathematical programming), one may verify that each equation in Lemma \ref{lem13421423a} holds as well as ${\bf P_1}$ and ${\bf P_2}$.
\end{proof}

\section{The remaining cases}

In this section, we compute the joint distribution of the first letter and descents statistics in the remaining cases.  In order to do so, we consider a common approach based on consideration of the second letter.  Given $n \geq 2$ and $i,j \in [n]$ with $i \neq j$, let $\mathcal{S}_{n,i,j}(\sigma,\tau)$ denote the subset of $\mathcal{S}_n(\sigma,\tau)$ whose members start $i,j$.  Let $a_n(i,j)=a_n(i,j;q)$ be defined by
$$a_n(i,j)=\sum_{\pi \in \mathcal{S}_{n,i,j}(\sigma,\tau)}q^{\text{desc}(\pi)}, \qquad n \geq 2,$$
for the particular pattern pair $(\sigma,\tau)$ under consideration.  Given $n \geq 2$ and $1 \leq i \leq n$, let $a_n(i)=\sum_{j\neq i}a_n(i,j)$, with $a_1(1)=1$.  Further, let $a_n=\sum_{i=1}^na_n(i)$ for $n \geq 1$.  Note that $a_n$ is itself a $q$-generalization of the Schr\"{o}der numbers.

In this section, we determine the distribution $a_n(i,j)$ in the cases where

$$(\sigma,\tau)=(1324,1342),(1243,1423),(1243,1342),(1243,1324).$$  Note that for all the pattern pairs under consideration, we have
\begin{align}\label{eqneg}
a_n(i,j)=qa_{n-1}(j), \quad 1 \leq j<i\leq n.
\end{align}

So the task remains to write a recurrence for $a_n(i,j)$ when $j>i$.  We observe the following obvious initial conditions for $1 \leq n \leq 3$, which apply to all of the pattern pairs: $a_1=a_1(1)=1$, $a_2(1,2)=1$, $a_2(2,1)=q$, $a_3(1,2)=1$, $a_3(1,3)=a_3(2,1)=a_3(2,3)=a_3(3,1)=q$ and $a_3(3,2)=q^2$.

In all the cases, it is convenient to convert the recurrences to generating functions as follows.  Define $A(x,v,w)=A(x,v,w;q)$ by $$A(x,v,w)=\sum_{n\geq2}\sum_{i=1}^n\sum_{j=1}^na_n(i,j)v^iw^jx^n,$$
with
$$A^+(x,v,w)=\sum_{n\geq2}\sum_{i=1}^{n-1}\sum_{j=i+1}^na_n(i,j)v^iw^jx^n \quad \text{and}\quad A^-(x,v,w)=\sum_{n\geq2}\sum_{i=2}^n\sum_{j=1}^{i-1}a_n(i,j)v^iw^jx^n.$$
Note that $A(x,v,w)=A^+(x,v,w)+A^-(x,v,w)$, by the definitions.  Further auxiliary generating functions that are case dependent will be defined below.

\subsection{The case $(1324,1342)$.}

The $a_n(i,j)$ when $i<j$ are given recursively for the pattern pair $(1324,1342)$ as follows.

\begin{lemma}\label{L1324x1342}
We have
\begin{align}
a_n(i,n)&=qa_{n-1}(i,n-1)+q\sum_{j=1}^{i}a_{n-2}(j), \qquad 1 \leq i \leq n-3,\label{1324x1342qrec1}
\end{align}
and
\begin{align}
a_n(n-2,n)&=q^2a_{n-3}+q\sum_{j=1}^{n-3}a_{n-2}(j), \quad n \geq 4,\label{1324x1342qrec2}
\end{align}
with $a_n(i,i+1)=a_{n-1}(i)$ for $1 \leq i \leq n-1$ and $a_n(i,j)=0$ for  $2 \leq i+1<j<n$.
\end{lemma}
\begin{proof}
Let $\mathcal{A}_{n,i,j}=\mathcal{S}_{n,i,j}(1324,1342)$ and $\mathcal{A}_{n,i}=\mathcal{S}_{n,i}(1324,1342)$.  Note that the formula for $a_n(i,i+1)$ when $i<n$ and for $a_n(i,j)$ when $i+1<j<n$ may be verified using the definitions.  To show \eqref{1324x1342qrec1}, suppose $\pi \in \mathcal{A}_{n,i,n}$, where $1 \leq i \leq n-3$, and let $k$ denote the third letter of $\pi$.  Then we must have $k=n-1$, $k=i+1$ or $k<i$, for $i<k<n-1$ would ensure an occurrence of $1324$ or $1342$.  If $k=n-1$, then we may remove $n$ resulting in a member of $\mathcal{A}_{n-1,i,n-1}$, which accounts for the first term on the right side of \eqref{1324x1342qrec1}.  Note that the factor of $q$ accounts for the descent occurring between $n$ and $n-1$.  If $k<i$ or $k=i+1$, then both $i$ and $n$ may be deleted resulting in a member of $\cup_{j=1}^i\mathcal{A}_{n-2,j}$. This is accounted for by $q\sum_{j=1}^i a_{n-2}(j)$, the factor of $q$ arising due to the descent occurring between the second and third letters within the enumerated permutations.  Combining this case with the previous yields \eqref{1324x1342qrec1}.  To show \eqref{1324x1342qrec2}, consider separately the cases when $k=n-1$ or $k \leq n-3$.  If $k=n-1$, then the first three letters $n-2,n,n-1$ may be deleted, resulting in $\sigma \in \mathcal{S}_{n-3}(1324,1342)$ with no restrictions.  This yields the $q^2a_{n-3}$ term on the right side, as there is a descent between $n$ and $n-1$ and also between $n-1$ and the first letter of $\sigma$ (note $n \geq 4$ implies $\sigma$ is nonempty).  If $k \leq n-3$, then the first two letters may be deleted, which accounts for $q\sum_{k=1}^{n-3}a_{n-3}(k)$ and completes the proof of \eqref{1324x1342qrec2}.
\end{proof}

By Lemma \ref{L1324x1342}, one has for example when $n=4$ the following array:
\begin{align*}
a_4(1,2)=1+q && a_4(1,3)=0 && a_4(1,4)=q+q^2\\
a_4(2,1)=q+q^2 && a_4(2,3)=2q && a_4(2,4)=q+q^2\\
a_4(3,1)=q+q^2 && a_4(3,2)=2q^2 && a_4(3,4)=q+q^2\\
a_4(4,1)=q+q^2 && a_4(4,2)=2q^2 &&a_4(4,3)=q^2+q^3,
\end{align*}
which may be verified using the definitions.

We now express the recurrences in Lemma \ref{L1324x1342} in terms of $A(x,v,w)$ and $A^{\pm}(x,v,w)$ as given above. Further, in this case, it is convenient to define the additional generating functions $$C(x,v)=\sum_{n\geq2}\sum_{i=1}^{n-1}a_n(i,i+1)v^ix^n, \quad D(x,v)=\sum_{n\geq2}\sum_{i=1}^{n-1}a_n(i,n)v^ix^n.$$ Then Lemma \ref{L1324x1342}, together with \eqref{eqneg} leads to the following system of functional equations:
\begin{align*}
A^-(x,v,w)&=v^2wqx^2+\frac{vqx}{1-v}(A(x,vw,1)-vA(vx,w,1)),\\
C(x,v)&=vx^2+xA(x,v,1),\\
D(x,v)&=qxD(x,v)-vq^2x^3A(vx,1,1)-v^2q^2x^4+\frac{qx^2}{1-v}(A(x,v,1)-A(vx,1,1))\\
&+qx^2A(vx,1,1)+vqx^2(A(vx,1,1)+vx)+vx^2,\\
A^+(x,v,w)&=wC(x,vw)+D(wx,v)-vw^2qx^2A(vwx,1,1)-vw^2x^2-v^2w^3qx^3.
\end{align*}

Eliminating $C(x,v)$ and $D(x,v)$ from the preceding system, and using the fact $A(x,v,w)=A^+(x,v,w)+A^-(x,v,w)$,  implies that $A(x,v,w)$ satisfies
\begin{align}\label{eq1324x1342qvw}
A(x,v,w)&=\left(wx+\frac{vqx}{1-v}\right)A(x,vw,1)-\frac{v^2qx}{1-v}A(vx,w,1)
+\frac{vw^2qx^2}{(1-v)(wqx-1)}A(vwx,1,1)\\
&-\frac{w^2qx^2}{(1-v)(wqx-1)}A(wx,v,1)
+\frac{(vwq^2x-vq-w)vwx^2}{wqx-1}.\nonumber
\end{align}
Taking $w=1$ in \eqref{eq1324x1342qvw}, and replacing $x$ by $x/v$, we obtain
\begin{align}\label{eq1324x1342qvw1}
&\left(1-\frac{x}{v}-\frac{qx}{1-v}+\frac{qx^2}{v(1-v)(qx-v)}\right)A(x/v,v,1)\\
&=\frac{(q^2x-vq-1)x^2}{qx-v}-\left(\frac{vqx}{1-v}-\frac{qx^2}{(1-v)(qx-v)}\right)A(x,1,1).\nonumber
\end{align}
Taking $v=v_0=\frac{1+x+t(x)}{2}$ in \eqref{eq1324x1342qvw1}, where $t(x)$ is as above, then cancels out the left-hand side of \eqref{eq1324x1342qvw1} and leads to
$$A(x,1,1)=\frac{-1+(1+2q)x+2q(1-q)x^2+t(x)}{2q(qx-x-1)}.$$

Thus, by \eqref{eq1324x1342qvw1}, we obtain an explicit formula for $A(x,v,1)$.  Substituting the expressions for $A(x,v,1)$ and $A(x,1,1)$ into \eqref{eq1324x1342qvw} then yields the following result.

\begin{theorem}\label{Th11324x1342}
The generating function $A(x,v,w)$ which enumerates members of $\mathcal{S}_n(1324,1342)$ for $n \geq 2$ according to the joint distribution of the first and second letter statistics and the number of descents is given by \eqref{eq1324x1342qvw}, where
\begin{align*}
A(x,v,1)&=\frac{xv(vqx-v-x)t(xv)}{2(vq^2x^2-vqx^2-qx+vx-v-x+1)(vqx-vx-1)}\\
&+\frac{(x+(qx^2+qx+3x^2-2x-1)v)vx}{2(vqx-vx-1)(vqx-1)(vq^2x^2-vqx^2-qx+vx-v-x+1)}\\
&-\frac{(2q^2x^2+3q^2x-qx^2-qx-3q+2x-1+(1-2q)(qx-1)vqx)v^3x^2}{2(vqx-vx-1)(vqx-1)(vq^2x^2-vqx^2-qx+vx-v-x+1)}+\frac{v^2qx^2}{1-vqx}.
\end{align*}
\end{theorem}

Taking $w=1$ in Theorem \ref{Th11324x1342}, and comparing results with $U(x,v,1)$ above, yields the following equality, which implies the desired equivalence of distributions.

\begin{theorem}
We have
$$A(x,v,1)+vx=U(x,v,1)+\frac{vx}{1-vqx}.$$
\end{theorem}

\subsection{The case $(1243,1423)$.}

In this and the next two subsections, it is demonstrated that $A(x,v,1)$ is given by the same formula as in Theorem \ref{Th11324x1342} above for each pattern pair under consideration.  Note however that the distribution of the second letter statistic will be seen to be distinct in each case.  For $(1243,1423)$ (and the subsequent two cases), we extend the recurrence for $a_n(i,j)$ when $i<j$ found in \cite{MaSh2} by introducing a further variable $q$ which tracks the number of descents.

This yields the following recurrence for $a_n(i,j)$ when $i<j$.

\begin{lemma}\label{12431423L1q}
We have
\begin{equation}\label{12431423e1q}
a_n(i,i+1)=a_{n-1}(i,i+1)+\sum_{a=1}^{i-1}\sum_{c=0}^{i-a-1}\sum_{b=a+1}^{i-c}q^{c+1}\binom{i-a-1}{c}a_{n-c-2}(a,b), \qquad 1 \leq i \leq n-2,
\end{equation}
with $a_n(n-1,n)=qa_{n-2}$,
\begin{align}
a_n(i,i+2)&=qa_{n-1}(i,i+1)+a_{n-1}(i,i+2)\notag\\
&\quad+\sum_{a=1}^{i-1}\sum_{c=0}^{i-a-1}\sum_{b=a+1}^{i-c+1}q^{c+1}\binom{i-a-1}{c}a_{n-c-2}(a,b), \qquad 1 \leq i \leq n-3,\label{12431423e2q}
\end{align}
with $a_n(n-2,n)=qa_{n-2}$, and
\begin{align}
a_n(i,j)&=qa_{n-1}(i,j-1)+\sum_{a=1}^{i-1}\sum_{c=0}^{i-a-1}q^{c+1}\binom{i-a-1}{c}a_{n-c-2}(a,j-c-2)\notag\\
&\quad+(1-\delta_{j,n})\cdot\left(a_{n-1}(i,j)+\sum_{a=1}^{i-1}\sum_{c=0}^{i-a-1}q^{c+1}\binom{i-a-1}{c}a_{n-c-2}(a,j-c-1)\right) \label{12431423e3q}
\end{align}
for $4 \leq i+3 \leq j \leq n$.
\end{lemma}

Define in this case $$C(x,v)=\sum_{n\geq2}\sum_{i=1}^{n-1}a_n(i,i+1)v^ix^n, \quad D(x,v)=\sum_{n\geq3}\sum_{i=1}^{n-2}a_n(i,i+2)v^ix^n.$$ Then Lemma \ref{12431423L1q}, together with \eqref{eqneg}, yields the following system:
\begin{align}
A^-(x,v,w)&=qv^2wx^2+\frac{qvx}{1-v}A(x,vw,1)-\frac{qv^2x}{1-v}A(vx,w,1),\notag\\
C(x,vw)&=qvwx^2A(vwx,1,1)+vwx^2+qv^2w^2x^3+xC(x,vw)\label{C(x,v)eqn0}\\
&+\frac{qx^2}{qvwx+vw-1}\big(vwA^+(\frac{vwx}{1-qvwx},1-qvwx,1)\notag\\
&-(1-qvwx)A^+(x,1-qvwx,\frac{vw}{1-qvwx})\big),\notag\\
D(x,vw)&=qx(1-x)C(x,vw)+xD(x,vw)\label{D(x,v)equn}\\ &+\frac{qx^2}{vw(qvwx+vw-1)}\big(v^2w^2A^+(\frac{vwx}{1-qvwx},1-qvwx,1)\notag\\
&-(1-qvwx)^2A^+(x,1-qvwx,\frac{vw}{1-qvwx})\big),\notag\\
A^+(x,v,w)&=w(1-qwx)(1-x)C(x,vw) +w^2(1-x)D(x,vw) +(1+qw)xA^+(x,v,w)\notag\\ &+\frac{qw^2x^2}{qvwx+v-1}((1-qvwx)A^+(x,1-qvwx,\frac{vw}{1-qvwx})-vA^+(x,v,w))\notag\\ &+\frac{qwx^2}{v(qvwx+v-1)}((1-qvwx)^2A^+(x,1-qvwx,\frac{vw}{1-qvwx})-v^2A^+(x,v,w)).\notag
\end{align}
Note that by \eqref{C(x,v)eqn0} and \eqref{D(x,v)equn}, one can obtain a formula for $C(x,vw)$ and $D(x,vw)$ in terms of $A$ and $A^+$.

By programming, one may verify that the solution of the preceding system of functional equations is given by the following result.

\begin{theorem}\label{Th1423x1243}
Let $r(x)=\sqrt{1-x}\sqrt{4q(q-1)v^2w^2x^2-4qvwx-x+1}$.
Then the generating function $A(x,v,w)$ is given by $A^+(x,v,w)+A^-(x,v,w)$, where
{\footnotesize\begin{align*}
&A^+(x,v,w)\\
&=\frac{vwx^2(1-w^2)(1-x)r(x)t(vxw)}
{4((q-1)x+1)((q-1)vwx-1)(q(q-1)vw^2x^2-qwx+(v-1)(x-1))(q(q-1)vwx^2-qx+(vw-1)(x-1))}\\
&+\frac{wx^2(2q(q-1)v^2w^2x^2-2qvwx+(2v^2-1)(x-1))(2q(q-1)v^2w^2x^2-2qvwx+(2v^2w^2-1)(x-1))t(vwx)}
{8v((q-1)x+1)((q-1)vwx-1)(q(q-1)vw^2x^2-qwx+(v-1)(x-1))(q(q-1)vwx^2-qx+(vw-1)(x-1))}\\
&+\frac{vwx^2(w^2-1)((2q^2-2q+1)vwx^2-vwx-2qx-x+1)r(x)}{
4((q-1)x+1)((q-1)vwx-1)(q(q-1)vw^2x^2-qwx+(v-1)(x-1))(q(q-1)vwx^2-qx+(vw-1)(x-1))}\\
&+\frac{(2q(q-1)v^2w^2x^2-2qvwx+(2v^2w^2-1)(x-1))\alpha x}
{8v((q-1)x+1)((q-1)vwx-1)(qvwx-vw-1)(q(q-1)vw^2x^2-qwx+(v-1)(x-1))(q(q-1)vwx^2-qx+(vw-1)(x-1))},\\
&A^-(x,v,w)\\
&=\frac{(x(qvx+vx-v-1)-x(q(q-1)v^2x^2+v^2x-v^2+vx-1)w+v(qx-1)(qvx-1)w^2)qv^2wx^2t(vwx)}{2((q-1)vwx-1)(q(q-1)v^2wx^2-qvx+(vx-1)(w-1))(q(q-1)vwx^2-qx+(vw-1)(x-1))}\\
&+\frac{(2vw-3vw^2+2w-2+((v+1)(2q+1)-(2qv^2+v^2+2q+2v+1)w+v(v+1)(q+2)w^2+v^2(2q-1)w^3)x)qv^2wx^2}{2((q-1)vwx-1)(q(q-1)v^2wx^2-qvx+(vx-1)(w-1))(q(q-1)vwx^2-qx+(vw-1)(x-1))}\\
&+\frac{((q+1)(2q+1)+2q(v+1)(q-2)w+((2q-2q^2-1)v^2+q(q+2)v-2q^2+2q-1)w^2+qv(v+1)(2q-1)w^3)qv^3wx^4}{2((q-1)vwx-1)(q(q-1)v^2wx^2-qvx+(vx-1)(w-1))(q(q-1)vwx^2-qx+(vw-1)(x-1))}\\
&+\frac{(4q^3-q^2-2q+1-(2q^2-2q+1)(v+1)w+q^2v(2q-1)w^2-(q-1)(2q^2-2q+1)qvwx)qv^4w^2x^5}{2((q-1)vwx-1)(q(q-1)v^2wx^2-qvx+(vx-1)(w-1))(q(q-1)vwx^2-qx+(vw-1)(x-1))},
\end{align*}}
and
$\alpha=4v^2w^2-6v^3w^2-2v^2w+vw^2+2v+w-2
+(2-2v-(2qv^2-6qv-4v^2+2v+1)w+v(6qv^2+8v^2-5q-4v-2)w^2+v^2(4qv^2-4qv-2v^2-6q+4v-1)w^3)x
+(2(v-1)(q-1)+(8q(1-q)v-6qv^2-2v^2+3q+1)w-v(2(2q+1)(q-1)v^2-10q^2-(q-1)(4v-1))w^2+2qv^2(5q-4)w^3)vwx^2
+(4q^2v+2(2q-1)v^2+(4v+1)(1-2q)-2v(5q^2-3q-1)w-2v^2(2q-1)(q-1)w^2)qv^2w^3x^3
+2q^2v^4w^5(2q-1)(q-1)x^4$.\medskip

Moreover, $C(x,v)$ and $D(x,v)$ can be found by equations \eqref{C(x,v)eqn0} and \eqref{D(x,v)equn}, respectively.
\end{theorem}

In particular, taking $w=1$ in the preceding theorem gives
\begin{align*}
A(x,v,1)&=\frac{xv(vqx-v-x)t(xv)}{2(vq^2x^2-vqx^2-qx+vx-v-x+1)(vqx-vx-1)}\\
&+\frac{(x+(qx^2+qx+3x^2-2x-1)v)vx}{2(vqx-vx-1)(vqx-1)(vq^2x^2-vqx^2-qx+vx-v-x+1)}\\
&-\frac{(2q^2x^2+3q^2x-qx^2-qx-3q+2x-1+(1-2q)(qx-1)vqx)v^3x^2}{2(vqx-vx-1)(vqx-1)(vq^2x^2-vqx^2-qx+vx-v-x+1)}+\frac{v^2qx^2}{1-vqx},
\end{align*}
which agrees with the comparable generating function formula found above in the case $(1324,1342)$.

\subsection{The case $(1243,1342)$.}

We have in this case the following recurrence for $a_n(i,j)$ when $i<j$.

\begin{lemma}\label{12431342L1q}
If $1 \leq i <j \leq n-1$, then
\begin{align}\label{12431342L1e1q}
a_n(i,j)&=q\sum_{k=i+1}^{j-1}a_{n-1}(i,k)+\sum_{a=1}^{i-1}\sum_{c=0}^{i-a-1}\sum_{b=a+1}^{j-c-2}q^{c+1}\binom{i-a-1}{c}a_{n-c-2}(a,b)\notag\\
&\quad+\delta_{i+1,j}\cdot\left(a_{n-1}(i,i+1)+\sum_{a=1}^{i-1}\sum_{c=0}^{i-a-1}q^{c+1}\binom{i-a-1}{c}a_{n-c-2}(a,i-c)\right),
\end{align}
with $a_n(i,n)=\sum_{j=1}^{i-1}a_{n-1}(i,j)+q\sum_{j=i+1}^{n-1}a_{n-1}(i,j)$ for $1 \leq i \leq n-1$.
\end{lemma}

Define $C(x,v)=\sum_{n\geq3}\sum_{i=1}^{n-1}a_n(i,i+1)v^ix^n$. Then Lemma \ref{12431342L1q}, together with \eqref{eqneg}, yields the following system:
\begin{align}
A^-(x,v,w)&=qv^2wx^2+\frac{vxq}{1-v)}A(x,vw,1)-\frac{v^2xq}{1-v}A(vx,w,1),\notag\\
A^+(x,v,w)&=vw^2x^2+wxA^-(wx,v,1)+qwxA^+(wx,v,1)+\frac{qwx}{1-w}(A^+(x,v,w)-A^+(wx,v,1))\notag\\
&+\frac{qv^2w^2x^2}{(qvwx+v-1)(qvwx+vw-1)}(A^+(\frac{vwx}{1-qvwx},1-qvwx,1)
-A^+(x,1-qvwx,\frac{vw}{1-qvwx}))\notag\\
&+\frac{qvw^2x^2}{(qvwx+v-1)(1-w)}(A^+(wx,v,1)-A^+(x,v,w))
+wC(x,vw)-qv^2w^3x^3\notag\\
&-qvw^2x^2A(vwx,1,1)\notag\\
&-\frac{qvw^2x^2}{qvwx+vw-1}(A^+(\frac{vwx}{1-qvwx},1-qvwx,1)
-A^+(x,1-qvwx,\frac{vw}{1-qvwx})),\notag\\
C(x,v)&=qv^2x^3+qvx^2A(vx,1,1)\label{C(x,v)eqn}\\
&+\frac{qvx^2}{qvx+v-1}(A^+(\frac{vx}{1-qvx},1-qvx,1)-A^+(x,1-qvx,\frac{v}{1-qvx}))\notag\\   &+xC(x,v)+vx^3  +qx^2A^+(x,1-qvx,\frac{v}{1-qvx}). \notag
\end{align}
Note that by \eqref{C(x,v)eqn} we have a formula for $C(x,v)$ in terms of $A$ and $A^+$.

By programming, one can show the solution of the preceding system is given by the following.

\begin{theorem}\label{Th1342x1243}
The generating function $A(x,v,w)$ is given by $A^+(x,v,w)+A^-(x,v,w)$, where
{\footnotesize\begin{align*}
&A^+(x,v,w)\\
&=\frac{(2v^2w+v-1+w(2qvw-2v^2w+qv-v^2-2v+1)x
-vw^2(2q^2vw-2qvw+2q-v-1)x^2+qv^2w^3(q-1)x^3)vw^2x^2t(vwx)}
{2((q-1)vwx^2-(2q-1)vwx+(2vw-1)(x-1))(q(q-1)vw^2x^2-qwx+(v-1)(wx-1))((q-1)vwx-1)}\\
&+\frac{(4vw-6v^2w+v-1+(1+(-2qw+3q-4w-5)v+(6-4qw-2q+10w)v^2+2w(2q-1)v^3)wx)vw^2x^2}
{2((q-1)vwx^2-(2q-1)vwx+(2vw-1)(x-1))(q(q-1)vw^2x^2-qwx+(v-1)(wx-1))((q-1)vwx-1)}\\
&+\frac{(4+(6q^2w-2q^2+7q-4w-9)v-(2w+1)(2q-1)v^2
-(4q^3vw-6q^2vw+2qvw+3q^2-2qv+q+v-3)vwx)v^2w^4x^4}
{2((q-1)vwx^2-(2q-1)vwx+(2vw-1)(x-1))(q(q-1)vw^2x^2-qwx+(v-1)(wx-1))((q-1)vwx-1)}\\
&+\frac{q(2q-1)(q-1)v^6w^6x^6}
{2((q-1)vwx^2-(2q-1)vwx+(2vw-1)(x-1))(q(q-1)vw^2x^2-qwx+(v-1)(wx-1))((q-1)vwx-1)},\\
&A^-(x,v,w)\\
&=\frac{(vw^2-(qv^2w^2+qvw^2-v^2w+v-w+1)x+(q^2vw^2-vw+q-w+1)vx^2-(q-1)qv^2wx^3)qv^2wx^2t(vwx)}{2(q(q-1)vwx^2-qx+(vw-1)(x-1))((q-1)vwx-1)(q(q-1)v^2wx^2-qvx+(vx-1)(w-1))}\\
&=\frac{(2vw-3vw^2+2w-2+((v+1)(2q+1)-((2q+1)v^2+2(q+v)+1)w
+v(v+1)(q+2)w^2+v^2(2q-1)w^3)x)qv^2wx^2}{2(q(q-1)vwx^2-qx+(vw-1)(x-1))((q-1)vwx-1)(q(q-1)v^2wx^2-qvx+(vx-1)(w-1))}\\
&-\frac{((1+q)(2q+1)+2q(v+1)(q-2)w-(2q^2v^2
-q^2v-2qv^2+2q^2-2qv+v^2-2q+1)w^2+qv(v+1)(2q-1)w^3)qv^3wx^4}{2(q(q-1)vwx^2-qx+(vw-1)(x-1))((q-1)vwx-1)(q(q-1)v^2wx^2-qvx+(vx-1)(w-1))}\\
&+\frac{(4q^3-q^2-2q+1-(2q^2-2q+1)(v+1)w+q^2v(2q-1)w^2
-wvq(q-1)(2q^2-2q+1)x)qv^4w^2x^5}{2(q(q-1)vwx^2-qx+(vw-1)(x-1))((q-1)vwx-1)(q(q-1)v^2wx^2-qvx+(vx-1)(w-1))}.\medskip
\end{align*}}

Moreover, $C(x,v)$ can be found by equation \eqref{C(x,v)eqn}.
\end{theorem}

Hence,
\begin{align*}
A(x,v,1)&=\frac{xv(vqx-v-x)t(xv)}{2(vq^2x^2-vqx^2-qx+vx-v-x+1)(vqx-vx-1)}\\
&+\frac{(x+(qx^2+qx+3x^2-2x-1)v)vx}{2(vqx-vx-1)(vqx-1)(vq^2x^2-vqx^2-qx+vx-v-x+1)}\\
&-\frac{(2q^2x^2+3q^2x-qx^2-qx-3q+2x-1+(1-2q)(qx-1)vqx)v^3x^2}{2(vqx-vx-1)(vqx-1)(vq^2x^2-vqx^2-qx+vx-v-x+1)}+\frac{v^2qx^2}{1-vqx},
\end{align*}
which agrees with the prior cases.

\subsection{The case $(1243,1324)$.}

There is the following recurrence for $a_n(i,j)$ when $i<j$.

\begin{lemma}\label{1243,1324q}
We have
\begin{equation}\label{i(i+1)recq}
a_n(i,i+1)=a_{n-1}(i,i+1)+\sum_{a=1}^{i-1}\sum_{b=a+1}^i\sum_{c=0}^{i-b}q^{c+1}\binom{i-a-1}{c}a_{n-c-2}(a,b),\quad 1 \leq i \leq n-2,
\end{equation}
and
\begin{equation}\label{ijrecq}
a_n(i,j)=a_{n-1}(i,j)+\sum_{a=1}^{i-1}\sum_{c=0}^{i-a-1}q^{c+1}\binom{i-a-1}{c}a_{n-c-2}(a,j-c-1), \quad 3 \leq i+2\leq j \leq n-1,
\end{equation}
with $a_n(i,n)=\sum_{j=1}^{i-1}a_{n-1}(i,j)+q\sum_{j=i+1}^{n-1}a_{n-1}(i,j)$ for $1 \leq i \leq n-1$.
\end{lemma}

Define $C(x,v)=\sum_{n\geq3}\sum_{i=1}^{n-1}a_n(i,i+1)v^ix^n$. Then Lemma \ref{1243,1324q}, together with \eqref{eqneg}, yields the following system:
\begin{align}
A^-(x,v,w)&=qv^2wx^2+\frac{qvx}{1-v}A(x,vw,1)-\frac{qv^2x}{1-v}A(vx,w,1),\notag\\
C(x,v)&=qv^2x^3+qvx^2A(vx,1,1)+xC(x,v)+vx^3\label{C(x,v)eqn2}\\
&+\frac{qx^2}{qvx+v-1}(vA^+(\frac{vx}{1-qvx},1-qvx,1)-(1-qvx)A^+(x,1-qvx,\frac{v}{1-qvx})),\notag\\
A^+(x,v,w)&=wxA^-(wx,v,1)+qwxA^+(wx,v,1)+vw^2x^2-qv^2w^3x^3-vw^2x^3+wC(x,vw)\notag\\
&-qvw^2x^2A(vwx,1,1)+xA^+(x,v,w)-wxC(x,vw)\notag\\
&+\frac{qwx^2}{qvwx+v-1}((1-qvwx)A^+(x,1-qvwx,\frac{vw}{1-qvwx})-vA^+(x,v,w)).\notag
\end{align}
Note that by \eqref{C(x,v)eqn2} we have a formula for $C(x,v)$ in terms of $A$ and $A^+$.

By programming, one can show the solution of the preceding system is given by the following.

\begin{theorem}\label{Th1324x1243}
The generating function $A(x,v,w)$ is given by $A^+(x,v,w)+A^-(x,v,w)$, where
{\footnotesize\begin{align*}
&A^+(x,v,w)\\
&=\frac{(2v^2w+v-1+w(v(q-v)(2w+1)-2v+1)x-vw^2(2q(q-1)vw+2q-v-1)x^2+qw^3v^2(q-1)x^3)vw^2x^2t(vwx)}
{2((q-1)vwx-1)((q-1)vwx^2-(2q+1)vwx+(2vw-1)(x-1))(q(q-1)vw^2x^2-qwx+(wx-1)(v-1))}\\
&+\frac{(4vw-6v^2w+v-1+(2(3-q)v^2+3qv-5v+1+2(2qv^2-2qv-v^2-q+5v-2)vw)wx)vw^2x^2}
{2((q-1)vwx-1)((q-1)vwx^2-(2q+1)vwx+(2vw-1)(x-1))(q(q-1)vw^2x^2-qwx+(wx-1)(v-1))}\\
&+\frac{q(7-2q)v+(1-2q)v^2-9v+4+2(3q^2-2qv+v-2)vw)v^2w^4x^4}
{2((q-1)vwx-1)((q-1)vwx^2-(2q+1)vwx+(2vw-1)(x-1))(q(q-1)vw^2x^2-qwx+(wx-1)(v-1))}\\
&+\frac{(2q(3q-2q^2-1)vw-3q^2+2qv-q-v+3+(2q-1)(q-1)qvwx)v^3w^5x^5}
{2((q-1)vwx-1)((q-1)vwx^2-(2q+1)vwx+(2vw-1)(x-1))(q(q-1)vw^2x^2-qwx+(wx-1)(v-1))},\\
&A^-(x,v,w)\\
&=\frac{(vw^2-(qv^2w^2+qvw^2-v^2w+v-w+1)x+v(q^2vw^2-vw+q-w+1)x^2-qv^2w(q-1)x^3)qv^2wx^2t(vwx)}
{2(q(q-1)v^2wx^2-qvx+(vx-1)(w-1))((q-1)vwx-1)(q(q-1)vwx^2-qx+(vw-1)(x-1))}\\
&+\frac{(2vw-3vw^2+2w-2
+((v+1)(1+2q)-(2qv^2+v^2+2q+2v+1)w+v(v+1)(2+q)w^2+v^2(2q-1)w^3)x)qv^2wx^2}
{2(q(q-1)v^2wx^2-qvx+(vx-1)(w-1))((q-1)vwx-1)(q(q-1)vwx^2-qx+(vw-1)(x-1))}\\
&+\frac{((1+q)(1+2q)+2q(v+1)(q-2)w+(q^2v+2q(1-q)v^2-2q^2+2qv-v^2+2q-1)w^2+qv(v+1)(2q-1)w^3)qv^3wx^4}
{2(q(q-1)v^2wx^2-qvx+(vx-1)(w-1))((q-1)vwx-1)(q(q-1)vwx^2-qx+(vw-1)(x-1))}\\
&+\frac{(4q^3-q^2-2q+1-(2q^2-2q+1)(v+1)w+q^2v(2q-1)w^2-(q-1)(2q^2-2q+1)qvwx)qv^4w^2x^5}
{2(q(q-1)v^2wx^2-qvx+(vx-1)(w-1))((q-1)vwx-1)(q(q-1)vwx^2-qx+(vw-1)(x-1))}.\medskip
\end{align*}}

Moreover, $C(x,v)$ can be found by equation \eqref{C(x,v)eqn2}.
\end{theorem}

Hence,
\begin{align*}
A(x,v,1)&=\frac{xv(vqx-v-x)t(xv)}{2(vq^2x^2-vqx^2-qx+vx-v-x+1)(vqx-vx-1)}\\
&+\frac{(x+(qx^2+qx+3x^2-2x-1)v)vx}{2(vqx-vx-1)(vqx-1)(vq^2x^2-vqx^2-qx+vx-v-x+1)}\\
&-\frac{(2q^2x^2+3q^2x-qx^2-qx-3q+2x-1+(1-2q)(qx-1)vqx)v^3x^2}{2(vqx-vx-1)(vqx-1)(vq^2x^2-vqx^2-qx+vx-v-x+1)}+\frac{v^2qx^2}{1-vqx},
\end{align*}
as in the prior cases.

\end{document}